\documentclass[11pt]{article}
\usepackage{amsmath, amssymb}
\usepackage[top=1in, bottom=1in, left=1in, right=1in]{geometry}
\usepackage{mathtools}
\usepackage{commath}
\usepackage{mathrsfs}
\usepackage{amsthm, color}
\usepackage{graphicx}
\usepackage{tikz}
\usetikzlibrary{cd,arrows.meta, automata,
                positioning,
                quotes}
\usepackage{graphicx}
\graphicspath{ {./} }
\usepackage{enumerate}
\usepackage{bbm}

\usepackage{tocloft}

\usepackage{hyperref}
\hypersetup{
    colorlinks,
    citecolor=black,
    filecolor=black,
    linkcolor=black,
    urlcolor=black
  }

\newtheorem{theorem}{Theorem}[section]
\newtheorem{remark}[theorem]{Remark}
\newtheorem{corollary}[theorem]{Corollary}
\newtheorem*{theorem*}{Theorem}
\newtheorem{lemma}[theorem]{Lemma}

\newtheorem{definition}[theorem]{Definition}
\newtheorem{proposition}[theorem]{Proposition}

\newtheorem{introtheorem}{Theorem}

\numberwithin{equation}{section}

\newcommand{\Img}[1]{\mathrm{Im}#1}
\DeclareMathOperator{\dist}{\mathrm{dist}}
\DeclareMathOperator{\id}{\mathrm{id}}

\DeclareMathOperator{\dbar}{\bar{\partial}}
\DeclareMathOperator{\ev}{\mathrm{ev}}

\DeclareMathOperator{\Coh}{\mathrm{Coh}}

\DeclareMathOperator{\Par}{\mathrm{Par}}

\DeclareMathOperator{\mwh}{\widetilde{\mathfrak{h}}}

\title{
Equivalence between Gromov-Witten invariants: Domain dependent perturbation and Kuranishi structure approaches
}
\author{Yuguo Qin\thanks{Email: qinyuguo@hu-berlin.de}}
\date{}

\begin{document}
\maketitle

\begin{abstract}
  We prove that the homology class induced by the rational pseudocycle constructed via domain-dependent perturbations by Cieliebak and Mohnke \cite{CM} coincides with the homology class induced by the virtual fundamental class defined through Kuranishi structures by Fukaya, Oh, Ohta, and Ono \cite{FK,FOOO}.
\end{abstract}


\section{Introduction}
Let $(M,\omega)$ be a closed symplectic manifold with a compatible almost complex structure $J$. The Gromov-Witten invariants of $(M,\omega)$, which encode counts of $J$-holomorphic curves, can be constructed through several different approaches. These include the Kuranishi structure framework established by Fukaya, Ono, Oh, and Ohta \cite{FK,FOOO}; domain-dependent perturbation techniques, applied to genus zero curves by Cieliebak and Mohnke \cite{CM} and extended to higher genus curves by Gerstenberger \cite{Gerstenberger} and Ionel, Parker \cite{IP}; the polyfold approach pioneered by Hofer, Wysocki, and Zehnder \cite{HWZ}; and the homotopy sheaf approach developed by Pardon \cite{Pardon}. This paper establishes the equivalence between the domain-dependent perturbation method and the Kuranishi structure approach.

For a homology class $A \in H_2(M;\mathbb{Z})$ and nonnegative integer $k$, the moduli space $\overline{\mathcal{M}}_{0,k}(J,A)$ of stable $J$-holomorphic genus zero curves with $k$ marked points consists of equivalence classes $[u,z_1,\ldots,z_k]$, where $u: \Sigma_T \to M$ is a $J$-holomorphic map from a genus zero nodal curve $\Sigma_T$ modeled on a tree $T$ to $M$ and represents the class $A$, $z_i\in \Sigma_T$ are marked points. The evaluation maps $\ev_i: \overline{\mathcal{M}}_{0,k}(J,A) \to M$ is defined by $\ev_i([u,z_1,\ldots,z_k]) = u(z_i)$.

In general, $\overline{\mathcal{M}}_{0,k}(J,A)$ may not be a smooth manifold, making it challenging to define a fundamental class. To address this issue, Cieliebak and Mohnke \cite{CM} use domain-dependent almost complex structures $\mathcal{J}_{\ell+1}$ that depend on the universal curve $\overline{\mathcal{M}}_{0,\ell+1}$ of genus zero curves with $\ell$ marked points, combined with a Donaldson hypersurface $Y \subset M$ of sufficiently large degree $D$. For a regular choice $K \in \mathcal{J}_{\ell+1}$ and $\ell = D\omega(A)$, they show that the evaluation map
\[\prod_{i=1}^k\ev_i: \mathcal{M}_{0,k+\ell}(K,A,Y) \to M^k\]
defines a pseudocycle, where $\mathcal{M}_{0,k+\ell}(K,A,Y)$ is the moduli space of stable $K$-holomorphic spheres with $k+\ell$ marked points and the last $\ell$ marked points are constrained to lie on the Donaldson hypersurface $Y$. Moreover, they prove that the rational pseudocycle $\frac{1}{\ell!}(\prod_{i=1}^k\ev_i)$ is independent of the auxiliary data $K$ and $Y$.

In contrast, Fukaya and Ono \cite{FK} construct a Kuranishi structure on $\overline{\mathcal{M}}_{0,k}(J,A)$ by providing local finite-dimensional smooth manifolds around the moduli space with obstruction bundles. Using multisections of these bundles, they define a virtual fundamental class $[s^{'-1}(0)]$ whose pushforward 
\[(\prod_{i=1}^k\ev_i)_{*}([s^{'-1}(0)])\]
gives a well-defined homology class in $M^k$.

Our main result establishes the equivalence between these two approaches:

\begin{introtheorem}[Theorem \ref{Thm:5.17}]\label{introcorollayA}
  For $K\in \mathcal{J}_{\ell+1}^{reg*}$ and $J_0\in B^{*}\cap \mathcal{J}_{\ell+1}^{*}(M,Y;J,\theta_1)$ connected by a path $\{K_t\}_{t\in[0,1]}\subset \mathcal{J}_{\ell+1}^{*}(M,Y;J,\theta_1)$ with $K_0=J_0$, $K_1=K$, take a generic multisection $s^{'}=\{s^{'}_w\}$ on the Kuranishi structure
  \[\{(V_w,E_{J_0,w},\Gamma_w,\psi_w,s_w)\}_{w\in \overline{\mathcal{M}}_{0,k}(J_0,A)}.\]
  Then the homology class given by the pseudocycle
  \[{\bf ev}_k\coloneqq \prod_{i=1}^k\ev_i:\mathcal{M}_{0,k+\ell}(K,A,Y)\rightarrow M^k\]
  is equal to $\ell !$ times the homology class ${\bf ev}_{k*}([s^{'-1}(0)])$.
\end{introtheorem}

Our proof strategy proceeds through the following key steps:

\begin{enumerate}
  \item We construct a Kuranishi structure on $\overline{\mathcal{M}}_{0,k+\ell}(K,A,Y)$ with additional obstruction data restricting the last $\ell$ marked points (Theorem \ref{Thm:5.12}).
  
  \item For regular $K \in \mathcal{J}_{\ell+1}^{reg*}$, we demonstrate that the homology class given by the pseudocycle defined on $\mathcal{M}_{0,k+\ell}(K,A,Y)$ equals the homology class given by the virtual fundamental class provided by the Kuranishi structure on the compactified moduli space $\overline{\mathcal{M}}_{0,k+\ell}(K,A,Y)$ (Theorem \ref{Thm:5.14}).
  
  \item For the $K$ above and domain independent $J_0$ (which can be regarded as a constant domain dependent almost complex structure) in the same connected component, we establish equivalence between virtual fundamental classes on $\overline{\mathcal{M}}_{0,k+\ell}(K,A,Y)$ and $\overline{\mathcal{M}}_{0,k+\ell}(J_0,A,Y)$ by constructing a cobordism between their respective Kuranishi structures (Theorem \ref{Thm:5.15}).
  
  \item We modify the obstruction data for maps in $\overline{\mathcal{M}}_{0,k}(J_0,A)$ to obtain obstruction data that vanishes near the Donaldson hypersurface, and pull it back to $\overline{\mathcal{M}}_{0,k+\ell}(J_0,A,Y)$ via forgetful maps (Section \ref{Sec:PullBack}).
  
  \item Finally, we prove that the virtual fundamental class provided by the Kuranishi structure on $\overline{\mathcal{M}}_{0,k+\ell}(J_0,A,Y)$ equals $\ell!$ times the virtual fundamental class on $\overline{\mathcal{M}}_{0,k}(J_0,A)$ (Theorem \ref{Thm:5.22}).
\end{enumerate}

The following diagram summarizes our approach:

\begin{center}
  \begin{scalebox}{0.8}{
    \begin{tikzpicture}[
      node distance = 4.5cm,
      every edge quotes/.style = {font=\scriptsize, auto, align=center}
      ]
      \node (0) [rectangle,draw,align=center, label={[font=\scriptsize]}] {Pseudocycle $\prod_{i=1}^k\ev_i$: \\$\mathcal{M}_{0,k+\ell}(K,A,Y)\to M^k$};
      \node (1) [rectangle,draw,right = of 0, align=center, label={[font=\scriptsize, align=center]above: Definition see \\{\bf Theorem \ref{Thm:5.12}}}] {Virtual fundamental class \\by {\bf 1st} Kuranishi structure \\on $\overline{\mathcal{M}}_{0,k+\ell}(K,A,Y)$};
      \node (2) [rectangle,draw,below =2cm of 1, align=center] {Virtual fundamental class \\by {\bf 1st} Kuranishi structure \\on $\overline{\mathcal{M}}_{0,k+\ell}(J_0,A,Y)$};
      \node (3) [rectangle,draw,left =4.5cm of 2, align=center, label={[font=\scriptsize, align=center]}] {Virtual fundamental class \\by {\bf 2nd} Kuranishi structure \\on $\overline{\mathcal{M}}_{0,k+\ell}(J_0,A,Y)$};
      \node (4) [rectangle,draw,below =2cm of 3, align=center, label={[font=\scriptsize, align=center]below: {\bf Theorem \ref{introcorollayA}}}] { $\ell!$ Virtual fundamental class \\by {\bf 2nd} Kuranishi structure \\on $\overline{\mathcal{M}}_{0,k}(J_0,A)$};
      \node (5) [rectangle,draw,below right =2cm and 4.5 of 3, align=center, label={[font=\scriptsize, align=center]below: {Section \ref{section:5.1}}}] {{\bf 2nd} Kuranishi structure \\on $\overline{\mathcal{M}}_{0,k}(J_0,A)$ vanishes \\near preimage of $Y$};

      \path[-Stealth]
      (0) edge["Compare via relative homology\\given by a compact subset","{\bf Theorem \ref{Thm:5.14}}" swap] (1)
      (1) edge["Cobordism between \\Kuranishi structures \\from $K$ to $J_0$", "{\bf Theorem \ref{Thm:5.15}}" swap] (2)
      (2) edge["On the same base space \\different Kuranishi structures \\induce the same virtual \\ fundamental class" swap] (3)
      (3) edge["Compare via relative homology\\given by a compact subset.\\ See {\bf Theorem \ref{Thm:5.22}}" swap] (4);
      \path[dashed,-Stealth] (5) edge["Pull back obstruction data\\and multisection via $\varphi_\ell$ \\defined in Section \ref{Sec:PullBack}" swap, bend left=20] (3);
    \end{tikzpicture}}
  \end{scalebox}
\end{center}

The paper is organized as follows: In Section \ref{Sec:3}, we review the pseudocycle construction given by Cieliebak and Mohnke in \cite{CM}. Section \ref{Sec:3} outlines the Kuranishi structure approach by Fukaya, Oh, Ohta, and Ono \cite{FK,FOOO}. Section \ref{Sec:4} contains the proofs of our main results.

{\bf Acknowledgements.}
I thank my supervisor Prof.Klaus Mohnke for his guidance and particularly for suggesting the modification of obstruction data near the Donaldson hypersurface. I am grateful to Prof.Chris Wendl for helpful comments, especially his question about torsion in homology groups that led me to discover Zinger's results in \cite{Zinger}. I also thank Prof.Wilhelm Klingenberg for his support during this research, the Elsa-Neumann-scholarship from the State of Berlin for their generous sponsorship.

\section{Pseudocycle by domain-dependent perturbation via Donaldson hypersurface}\label{Sec:3}
In this section, we introduce the pseudocycle construction given by Cieliebak and Mohnke in \cite{CM}. Denote by $\overline{\mathcal{M}}_{0,k+1}$ the Deligne-Mumford space of stable curves of genus $0$ with $k+1$ marked points. For a symplectic manifold $(M,\omega)$, the space of domain-dependent almost complex structures $\mathcal{J}_{\overline{\mathcal{M}}_{0,k+1}}$ consists of smooth sections of the pullback bundle $\mathbf{J}_{\overline{\mathcal{M}}_{0,k+1}}$:
\begin{equation*}
  \begin{tikzcd}
    \mathbf{J}_{\overline{\mathcal{M}}_{0,k+1}}\arrow[d] & \mathbf{J}\arrow[d]\\
    \overline{\mathcal{M}}_{0,k+1}\times M\arrow[r]&M
  \end{tikzcd}
\end{equation*}
, where $\mathbf{J}\to M$ is the bundle whose fiber at $x\in M$ is the space of $\omega$-compatible complex structures on the tangent space $T_xM$. 

The space $\mathcal{J}_{\overline{\mathcal{M}}_{0,k+1}}$ lacks the Banach manifold structure needed for transversality arguments. Following \cite[Sec.~3]{CM}, we define $\mathcal{J}_{k+1}$ as
\[\mathcal{J}_{k+1}\coloneqq \exp_{J}(\Coh(\overline{\mathcal{M}}_{0,k+1},T_{J}\mathcal{J}))\subset \mathcal{J}_{\overline{\mathcal{M}}_{0,k+1}},\]
where $T_{J}\mathcal{J}$ is the tangent space of $\mathcal{J}$ at $J$, $\Coh(\overline{\mathcal{M}}_{0,k+1},T_{J}\mathcal{J})$ consists of coherent maps from $\overline{\mathcal{M}}_{0,k+1}$ to $T_{J}\mathcal{J}$ with finite $\epsilon$-Floer norm and $\exp_{J}$ is the exponential map. The space $\mathcal{J}_{k+1}$ forms a Banach manifold.

The singular curves arising from the Gromov compactification of domain stable spheres may not remain domain stable, and thus may no longer be elements of $\overline{\mathcal{M}}_{0,k+1}$. To address this issue, the authors in \cite{CM} use the Donaldson hypersurface.

Assume that $(M, \omega)$ is of dimension $2n$ with $[\omega] \in H^2(M; \mathbb{Z})$. For sufficiently large integers $D$, there exist symplectic hypersurfaces $Y \subset M$ Poincaré dual to $D[\omega]$ that are $\bar{J}$-holomorphic for some $\omega$-compatible almost complex structure $\bar{J}$ arbitrarily $C^{0}$-close to $J$ \cite{SKD,CM}.

\begin{definition}[{\cite[Rmk.~8.10, Def.~9.2]{CM}}]\label{sym:DonaldsonPair}\label{Def:3.14}
  A {\bf Donaldson pair} of degree $D$ is a pair $(J,Y)$ consisting of an $\omega$-compatible almost complex structure $J$ and a hypersurface $Y\subset M$ that is Poincaré dual to $D[\omega]$ and satisfies
  \[\Theta(Y;\omega,J)<\theta_2,\quad D\geq D^{*}(M,\omega,J),\]
  where $\Theta(Y;\omega,J)$ denotes the supremum of the Kähler angles of the tangent spaces of $Y$, the constant $D^{*}(M,\omega,J)$ depends on $M$, $\omega$, $J$, and $\theta_0$, and $0<\theta_2<\theta_0<1$ are fixed constants.
\end{definition}

For a Donaldson pair $(J,Y)$ and a constant $\theta>0$, there exists an $\omega$-compatible almost complex structure $K$ such that $Y$ is $K$-holomorphic and $\|K(y)-J(y)\|<\theta$ for all $y\in Y$ (See \cite[Lem.~8.5,8.9]{CM}). Define 
\[\mathcal{J}(M,Y;J,\theta) \coloneqq \{K\in\mathcal{J} \mid K(TY)=TY, \|K(m)-J(m)\|<\theta, m\in M\} \]
where $\mathcal{J}$ is the space of $\omega$-compatible almost complex structures on $M$. For $K$-holomorphic maps $f:S^2\rightarrow M$ and $K$-complex hypersurface $Y$, the intersection $f^{-1}(Y)$ is finite with positive local intersection numbers $\iota(f,Y,z)>0$ at $z\in f^{-1}(Y)$.

Authors in \cite{CM} stabilize curves by constraining additional marked points to lie on $Y$, the following proposition shows that the intersecions are enough to stabilize the curves for a generic choice of the almost complex structure.

\begin{proposition}[{\cite[Prop.~8.13, Cor.~8.16]{CM}}]\label{Prop:3.16}
  For a Donaldson pair $(J,Y)$ of sufficiently large degree $D$ and appropriate energy bound $E$, there exists an open dense subset $\mathcal{J}^{*}(M,Y;J,\theta_1,E)\subset \mathcal{J}(M,Y;J,\theta_1)$ such that for $K\in\mathcal{J}^{*}(M,Y;J,\theta_1,E)$:
  \begin{enumerate}
    \item All $K$-holomorphic spheres of energy less than $E$ in $Y$ are constant.
    \item Every nonconstant $K$-holomorphic sphere of energy less than $E$ intersects $Y$ in at least 3 distinct points.
  \end{enumerate}
  $\theta_1$ is a constant satisfying $\theta_2<\theta_1<\theta_0$.
\end{proposition}

In the case of domain-dependent almost complex structures, the authors in \cite{CM} define the following space:

\begin{definition}[{\cite[Def.~9.4]{CM}}]\label{Def:3.20}
  For $\ell\geq 3$ and $E_\ell=\ell/D$, define the space of domain-dependent almost complex structures
  \[\mathcal{J}_{\ell+1}^{*}(M,Y;J,\theta_1)\coloneqq\{K\in\mathcal{J}_{\ell+1}\mid K(\zeta)\in B^{*},\forall \zeta\in\overline{\mathcal{M}}_{0,\ell+1}\},\]
  where $B^{*}\subset \mathcal{J}^{*}(M,Y;J,\theta_1,E_\ell)$ is a $\theta_2$-contractible open subset, that is, $B^{*}$ is contractible in $\mathcal{J}^{*}(M,Y;J,\theta_1,E_\ell)$ to a point of $\mathcal{J}^{*}(M,Y;J,\theta_2,E_\ell)$. $\mathcal{J}_{\ell+1}^{*}(M,Y;J,\theta_1)$ is a Banach manifold.
\end{definition}

Now, fix a Donaldson pair $(J,Y)$ of sufficiently large degree $D$ and a domain-dependent almost complex structure $K\in\mathcal{J}_{\ell+1}^{*}(M,Y;J,\theta_1)$. For an $\ell$-stable $(k+\ell)$-labeled tree $T$ and homology decomposition $A=\sum_{\alpha\in T} A_\alpha$, define
\[\widetilde{\mathcal{M}}_{0,T}(K,A,Y) = \{(u,{\bf z}) \mid \dbar_Ku=0, [u_\alpha]=A_\alpha, u(z_i)\in Y \text{ for }k+1\leq i\leq k+\ell\}.\]
Denote by $\mathcal{M}_{0,T}(K,A,Y)\coloneqq \widetilde{\mathcal{M}}_{0,T}(K,A,Y)/\Gamma_T$, where $\Gamma_T$ is the automorphism group of the underlying curve modeled on the tree $T$. Let $T_0$ be the $(k+\ell)$-labeled tree with only one node, we call $\mathcal{M}_{0,k+\ell}(K,A,Y)\coloneqq \mathcal{M}_{0,T_0}(K,A,Y)$ the top stratum of the moduli space. We call the following space
\[\overline{\mathcal{M}}_{0,k+\ell}(K,A,Y)=\coprod_{T}\mathcal{M}_{0,T}(K,A,Y)\]
the compactified moduli space. The main result of \cite{CM} is the following theorem:

\begin{theorem}[{\cite[Thm.~1.2,1.3]{CM}}]\label{Thm:3.25}
  For a Donaldson pair $(J,Y)$ of sufficiently large degree $D$ and $\ell=D\omega(A)$, there exists a Baire set $\mathcal{J}_{\ell+1}^{reg*}\subset \mathcal{J}_{\ell+1}^{*}(M,Y;J,\theta_1)$ such that for $K\in \mathcal{J}_{\ell+1}^{reg*}$, the evaluation map ${\bf ev}_k:\mathcal{M}_{0,k+\ell}(K,A,Y)\rightarrow M^k$ represents a pseudocycle. Moreover, the rational pseudocycle $\frac{1}{\ell!}{\bf ev}_k$ is independent of auxiliary data.
\end{theorem}

In this paper, we will use the homology class induced by the pseudocyle. In \cite{Zinger}, A.Zinger establishes a natural isomorphism between equivalence classes of pseudocycles and integral homology classes of a smooth manifold $X$, eliminating the need to divide by the torsion part. A key result from \cite{Zinger} is:

\begin{proposition}[{\cite[Prop.~2.2]{Zinger}}]\label{Prop:3.3.1}
  Let $W^{'}$ be an oriented smooth manifold, if $g:W^{'}\to X$ is a smooth map and $V$ is an open neighborhood of a subset $B$ of $\Img{g}$ in $X$, there exists a neighborhood $U$ of $B$ in $V$ such that
  \[H_l(U)=0,\quad\text{if }l>\dim W^{'}.\]
\end{proposition}

Consider the long exact sequence of homology groups associated to the pair $(X,U)$, if $\dim W^{'}\leq\dim W-2$, we have
\begin{equation}\label{eq:3.3.2}
  0\to H_{\dim W}(X;\mathbb{Z})\to H_{\dim W}(X,U;\mathbb{Z})\to 0.
\end{equation}

Denote by
\[\mathrm{BD}\overline{\mathcal{M}}_{0,k+\ell}(J,A,Y)\coloneqq \bigcup_T\mathcal{M}_{0,T}(J,A,Y)\]
the boundary part of $\overline{\mathcal{M}}_{0,k+\ell}(J,A,Y)$, where $T$ ranges over all $\ell$-stable $(k+\ell)$-labeled trees with more than one node. According to Proposition \ref{Prop:3.3.1}, we can take an open neighborhood $U$ of ${\bf ev}_k(\mathrm{BD}\overline{\mathcal{M}}_{0,k+\ell}(J,A,Y))$ in $M^k$ such that
\[H_l(U)=0,\quad\text{if }l>2(n-3+k+c_1(A))-2\]
and therefore, by (\ref{eq:3.3.2}),
\begin{equation}\label{eq:3.3.3}
  H_{2(n-3+k+c_1(A))}(M^k,U;\mathbb{Z})\cong H_{2(n-3+k+c_1(A))}(M^k;\mathbb{Z}).
\end{equation}

Now we can take a sufficiently large compact subset $\mathfrak{C}\subset \mathcal{M}_{0,k+\ell}(J,A,Y)$ such that the evaluation map
\[{\bf ev}_k:\mathfrak{C}\rightarrow M^k\]
represents a cycle in $H_{2(n-3+k+c_1(A))}(M^k,U;\mathbb{Z})$. By (\ref{eq:3.3.3}), this evaluation map represents a cycle in $H_{2(n-3+k+c_1(A))}(M^k;\mathbb{Z})$. The homology class $[{\bf ev}_k]$ defined in this manner is independent of the choice of compact subset $\mathfrak{C}$ and neighborhood $U$ (see the proof of \cite[Lem.~3.5]{Zinger}). The main result of \cite{Zinger} proves that 

\begin{theorem}[{\cite[Thm.~1.1]{Zinger}}]\label{Thm:3.3.1}
  If $X$ is a smooth manifold, and $\mathcal{H}_{*}(X)$ denotes the set of equivalence classes of pseudocycles into $X$, there exist natural homomorphisms of graded $\mathbb{Z}$-modules
  \begin{equation}\label{eq:3.3.4}
    \Psi_{*}:H_{*}(X;\mathbb{Z})\rightarrow \mathcal{H}_{*}(X)\text{ and }\Phi_{*}:\mathcal{H}_{*}(X)\rightarrow H_{*}(X;\mathbb{Z}),
  \end{equation}
  such that $\Psi_{*}\circ\Phi_{*}=\mathrm{Id}$ and $\Phi_{*}\circ\Psi_{*}=\mathrm{Id}$.
\end{theorem}

\section{Virtual fundamental class via Kuranishi structure construction}\label{Sec:4}
We present Fukaya, Oh, Ohta, Ono's construction of the virtual fundamental class on $\overline{\mathcal{M}}_{0,k}(J,A)$ via Kuranishi structures. For a metrizable compact space $X$, let $(V_u, E_u, \Gamma_u, \psi_u, s_u)$ be the Kuranishi neighborhood at $u\in X$, where $V_u$ is a smooth manifold of finite dimension (possibly with boundary or corner), $E_u$ is a finite-dimensional vector space, $\Gamma_u$ is a finite group acting smoothly and effectively on $V_u$ with a linear representation on $E_u$, $s_u: V_u \rightarrow E_u$ is a $\Gamma_u$-equivariant smooth map, and $\psi_u$ is a homeomorphism from $s^{-1}_u(0)/\Gamma_u$ to a neighborhood of $u$ in $X$. We also denote $\mathcal{E}_u: E_u \times V_u \rightarrow V_u$ the \textbf{obstruction bundle} and $U_u \coloneqq V_u/\Gamma_u$.

Assume that the Kuranishi structure $\{V_u,E_u,\Gamma_u,\psi_u,s_u\}_{u\in X}$ on $X$ admits the tangent bundle corresponds to the isomorphisims $\Phi_{uv}:N_{V_u}V_v\cong E_u/E_v$ for $u,v\in X$, where $N_{V_u}V_v$ is the normal bundle of $V_v\subset V_u$. Let $h=\{h_u\}$ be a multisection. If $h$ is transverse to $0$, then $h^{-1}(0)$ admits a smooth triangulation of dimension $d = \dim V_p-\dim E_p$ \cite[Lem.~4.2,6.9]{FK}. Write the trangulation as $h^{-1}(0)=\bigcup_{\aleph=1}^{\mathfrak{m}}\sigma_\aleph(\Delta_d)$, where $\sigma_\aleph$ is a diffeomorphism from the interior of the $d$-simplex $\Delta_d$ to a submanifold of some $V_u$.

Let $h_u$ be a local representative of $h$ on $V_u$ with branches $h_{u,i}:V_u\rightarrow E_u$, $i=1,\ldots,l_{h_u}$. Let $m_\aleph$ be a integer such that, $i_1,\ldots,i_{m_\aleph}$ is the set of all indices $i$ with $h_{u,i}(v)=0$ for $v\in \sigma_\aleph(\Delta_d)$. For $j=1,\ldots,m_\aleph$, we have an isomorphism induced by $Dh_{u,i_j}$
\begin{equation}
  \label{eq:4.1.1}
    \det TV_u\otimes \det \mathcal{E}_u^{*}\cong \det T\sigma_\aleph(\Delta_d).
\end{equation}
With trivializations of $\det TV_u\otimes \det \mathcal{E}_u^{*}$ and $\det T\sigma_\aleph(\Delta_d)$, define $\delta_j = +1$ if the isomorphism is orientation-preserving and $\delta_j = -1$ otherwise. The multiplicity of $\sigma_\aleph(\Delta_d)$ is then defined as (see \cite[Sec.~4]{FK}):
\begin{equation}
  \label{eq:1}
  mul_{\sigma_\aleph(\Delta_d)}=\sum_{j=1}^{m_\aleph}\delta_j/l_{h_u}.
\end{equation}

Let $Y$ be a topological space. For a {\bf strongly continuous map} \cite[Def.~6.6]{FK}, authors of \cite{FK} prove that
\[f_{*}(h^{-1}(0))\coloneqq\sum_\aleph mul_{\sigma_\aleph(\Delta_d)}(f_{u})_{*}([\sigma_\aleph(\Delta_d)])\]
gives a $\mathbb{Q}$-singular cycle in $Y$, which represents a homology class $[f_{*}(h^{-1}(0))]\in H_d(Y;\mathbb{Q})$. $[f_{*}(h^{-1}(0))]$ is called a {\bf virtual fundamental class}, and it is independent of the choice of the triangulation and $h$. If $X$ has an oriented $n$-dimensional Kuranishi structure with boundary $\partial X=X_1\cup X_2$ then:
  \[[(f|_{X_1})_{*}(h_1^{-1}(0))]=[(f|_{X_2})_{*}(h_2^{-1}(0))]\in H_{d-1}(Y;\mathbb{Q}),\]
where $h_1$ and $h_2$ are generic multisections on $X_1$ and $X_2$ respectively. This is called the {\bf cobordism between Kuranishi structures}.

The special case $X=\overline{\mathcal{M}}_{0,k}(J,A)$ is indeed metrizable and compact. For each $u\in \overline{\mathcal{M}}_{0,k}(J,A)$, denote by $\Sigma_u$ the underlying curve of $u$ and $\Gamma_u$ the group of automorphisms of $\Sigma_u$, such that $u\circ\gamma=u$ for $\gamma\in \Gamma_u$. As proved in \cite{FK}, there exists a finite-dimensional $\Gamma_u$-invariant subspace $E_{J,u}\subset W^{r-1,p}(\Omega^{0,1}(\Sigma_u,u^{*}TM))$ such that $E_{J,u}\oplus \Img D_u\bar{\partial}_J=W^{r-1,p}(\Omega^{0,1}(\Sigma_u,u^{*}TM))$. Moreover, Elements of $E_{J,u}$ are supported away from the singular points of $\Sigma_u$.

For $u\in \overline{\mathcal{M}}_{0,k}(J,A)$, authors of \cite{FK} construct a smooth manifold $V_u$ parametrizing maps near $u$, including those corresponding to different strata. Denote by $\Sigma_T$ the genus zero singular curve modeled on a $k$-labeled tree $T$, and $\mathcal{M}_{0,T}(J,A)$ the space of $J$-holomorphic maps from $\Sigma_T$ to $M$ modulo the action of the automorphism group $\Gamma_T$.

Let $\nabla$ be the Levi-Civita connection on $M$. For the complex linear connection $\widetilde{\nabla}\coloneqq \nabla-\frac{1}{2}J(\nabla J)$, denote by $\Par^{hol}_{x,y}:T_xM\to T_yM$ the parallel transport induced by $\widetilde{\nabla}$. Assume that $u,v\in\mathcal{M}_{0,T}(J,A)$, then the map $\Par^{hol}_{x,y}$ induces a map $\Par^{hol}_{u,v}:u^{*}TM\to v^{*}TM$. We transport obstruction spaces: for $v$ near $u$, define $E_{J,v}\coloneqq \Par^{hol}_{u,v}E_{J,u}$.

\begin{definition}\label{Def:4.22.2}
  Let $\mathcal{E}_{u,map}$ be the bundle over $V_u$ with fiber $E_{J,v}$ at $v\in V_u$. For $v$ with the same underlying curve as $u$, define
  \[V_{u,map}\coloneqq\{v\in V_{u} \mid \dbar_Jv\equiv 0\mod E_{J,v}\}.\]
\end{definition}

Next, we discuss maps across strata. Denote by $V_{u,resolve}$ a finite dimensional vector space parametrizing curves in $\overline{\mathcal{M}}_{0,k}$ near $\Sigma_T$ that are obtained by gluing sphere components of $\Sigma_T$ near its nodal points. For $\alpha,\beta\in \Sigma_T$, let $S_\alpha,S_\beta$ be sphere components of $\Sigma_T$. According to \cite{FK},
\[V_{u,resolve}=\bigoplus_{\substack{z=z_{\alpha\beta}=z_{\beta\alpha}\\ \text{ is a nodal point }}}T_{z_{\alpha\beta}}S_\alpha\otimes T_{z_{\beta\alpha}}S_\beta.\]
Denote by $V_{u,deform}$ a finite dimensional vector space parametrizing deformations of the marked points of $\Sigma_T$. According to \cite{FK},
\[0\in V_{u,deform}\subset \prod_{\Sigma_\alpha\text{ is stable}}\mathbb{C}^{k_\alpha-3},\]
where $k_\alpha$ is the number of marked points on the stable component $\Sigma_\alpha$. Overall, $V_{u,resolve}\times V_{u,deform}$ parametrizes a neighborhood of $\Sigma_T$ in $\overline{\mathcal{M}}_{0,k}$, and $0\in V_{u,resolve}\times V_{u,deform}$ corresponds to $\Sigma_T$. 

Recall that the sections in $E_{J,u}$ are supported on a compact subset $K_{ob} \subset \Sigma_T$ that is disjoint from the singular points of $\Sigma_T$. For each curve $\Sigma_{T^{'}}$ in the neighborhood parametrized by $V_{u,\mathrm{resolve}} \times V_{u,\mathrm{deform}}$, there exists a compact subset $K_{ob}^{'} \subset \Sigma_{T'}$ that can be identified with $K_{ob}$. This identification is induced by the extra marked points on $\Sigma_T$ and $\Sigma_{T^{'}}$ given by the {\bf stabilization data} as in \cite[Def.~17.7~(1)(8)]{FOOO-T}. Now assume that the map $v$ is from the underlying curve $\Sigma_{T^{'}}$ to $M$. We can transport $E_{J,u}$ to $v$ using the parallel transport $\Par^{hol}_{x,y}:T_xM\to T_yM$ together with the identification between $K_{ob}$ and $K_{ob}^{'}$. We denote this transported vector space by $E_{J,v}\coloneqq \Par^{hol}_{u,v}E_{J,u}$.

Now for $(\mathbf{v},\mathfrak{a},\vartheta)\in V_{u,map}\times V_{u,resolve}\times V_{u,deform}$, there is a unique map $v:\Sigma_{T^{'}}\rightarrow M$ such that the underlying curve $\Sigma_{T^{'}}\in\overline{\mathcal{M}}_{0,k}$ corresponds to the parameter $(\mathfrak{a},\vartheta)\in V_{u,resolve}\times V_{u,deform}$, and $\dbar_Jv=\dbar_J\mathbf{v}$ with respect to the identification between $K_{ob}$ and $K_{ob}^{'}$. We denote this unique correspondence by a map $\mu:(\mathbf{v},\mathfrak{a},\vartheta)\mapsto v$.  The full neighborhood $V_u$ is then given by
\[V_u\coloneqq \mu(V_{u,map}\times V_{u,resolve}\times V_{u,deform}).\]

Let $\mathcal{E}_{J,u}\rightarrow V_u$ be the vector bundle with fiber $E_{J,v}$ at $v\in V_u$, and we call $\mathcal{E}_{J,u}$ the {\bf obstruction bundle}. For $(\mathbf{v},\mathfrak{a},\vartheta)\in V_{u,map}\times V_{u,resolve}\times V_{u,deform}$, according to \cite[Prop.~12.23]{FK}, there exists a continuous section $s_u:V_u\rightarrow \mathcal{E}_{J,u}$ such that:
  \begin{enumerate}
    \item $s_u(\mu(\mathbf{v},\mathfrak{a},\vartheta))= \bar{\partial}_{J}\mu(\mathbf{v},\mathfrak{a},\vartheta)$
    \item $s_u$ is $\Gamma_u$-equivariant
    \item If $\varphi_\gamma:\Sigma_{\mathfrak{a},\vartheta}\rightarrow \Sigma_{\gamma(\mathfrak{a}),\gamma(\vartheta)}$ is the biholomorphic map induced by $\gamma\in\Gamma_u$, then $u_{\gamma(\mathbf{v}),\gamma(\mathfrak{a}),\gamma(\vartheta)}\circ\varphi_\gamma=u_{v,\mathfrak{a},\vartheta}$
    \item $s_u(\mu(0))=0$
  \end{enumerate}
  Moreover, $\psi_u:s_u^{-1}(0)/\Gamma_u\rightarrow\overline{\mathcal{M}}_{0,k}(J,A)$ is the identity map. We can perturb $s_u$ to a transverse multisection $s^{'}_u$. Then $(\prod_{i=1}^k\ev_i)_{*}([s^{'-1}(0)])$ defines a cycle in $H(M^k;\mathbb{Q})$, which is independent of $J$, the choice of generic multisection $s^{'}$, and other auxiliary choices. This class is called the {\bf virtual fundamental class}.

\section{The proof}\label{Sec:5}
This section will prove the main theorems. Let $(M,\omega)$ be a $2n$ dimensional smooth symplectic manifold, and $J$ an $\omega$-compatible almost complex structure, homology class $A\in H_2(M;\mathbb{Z})$ with $\omega(A)>0$\label{Fix:C:2:2}, $[\omega]\in H^2(M;\mathbb{Z})$. Let $(J,Y)$ be a Donaldson pair of sufficiently large degree $D$ with constants $0<\theta_2<\theta_0<1$ as in Definition \ref{Def:3.14}, with $\ell=D\omega(A)\geq 3$. Take a constant $\theta_1$ with $\theta_2<\theta_1<\theta_0$, and recall definition \ref{Def:3.20} of the space $\mathcal{J}_{\ell+1}^{*}(M,Y;J,\theta_1)$. We start by constructing a Kuranishi structure over the moduli space $\overline{\mathcal{M}}_{0,k+\ell}(K,A,Y)$.

\subsection{Kuranishi structure on $\overline{\mathcal{M}}_{0,k+\ell}(K,A,Y)$}\label{Sec:5.1}

Let $K\in \mathcal{J}_{\ell+1}^{*}(M,Y;J,\theta_1)$ and $T$ be a $\ell$ stable $(k+\ell)$-labeled tree. Denote by $\Sigma_T$ the genus zero curve modeled on $T$. Let $r$ be a positive integer, $p>1$ and $r-2/p>1$. Recall that $\dbar_K$ is a section of the bundle $\mathcal{E}^{r,p}(T,M)\rightarrow W^{r,p}(\Sigma_T,M)$ with fiber $W^{r-1,p}(\Omega^{0,1}(\Sigma_T,u^{*}TM))$ at $u\in W^{r,p}(\Sigma_T,M)$, and $\overline{\mathcal{M}}_{0,k+\ell}(K,A,Y)$ is stratified by
\[\overline{\mathcal{M}}_{0,k+\ell}(K,A,Y)=\coprod_{T}\mathcal{M}_{0,T}(K,A,Y).\]
Let $\ev_i$ be the evaluation map of $i$-th marked point, then
\[\widetilde{\mathcal{M}}_{0,T}(K,A,Y)=\left(\bar{\partial}_K\times\prod_{i=k+1}^{k+\ell}\ev_i\right)^{-1}( \{0\}\times Y),\]
and
\[\mathcal{M}_{0,T}(K,A,Y)=\widetilde{\mathcal{M}}_{0,T}(K,A,Y)/\Gamma_T.\]
For each $\ell$ stable $k+\ell$-labeled tree $T$ \label{Fix:M2:32}and each $(u,{\bf z})\in \widetilde{\mathcal{M}}_{0,T}(K,A,Y)$, take the obstruction data $E_{K,u}$ with respect to $\dbar_K$ as in Section \ref{Sec:4}. We do not need the symmetric stabilization data here, since the underlying curve of $u$ is already stable. Define spaces $V_{u,map}$, $V_{u,resolve}$, $V_{u,deform}$ and the map $\mu$ with respect to $\dbar_K$ as in Section \ref{Sec:4}. Denote by
\[V_u\coloneqq \mu(V_{u,map}\times V_{u,resolve}\times V_{u,deform}).\]
Then there is a vector bundle
\[\mathcal{E}_{K,u}\rightarrow V_u\]
whose fiber at $v\in V_u$ is the vector space $\Par_{u,v}^{hol}E_{K,u}$. Moreover, there is a map $s_u:V_u\rightarrow \mathcal{E}_{K,u}$ such that
\[s_u(v)=\dbar_K(v),\ v\in V_u.\]
However, for a $v\in V_u$, since the images of its last $\ell$ marked points does not neccesarily lie in $Y$, the set $s_u^{-1}(0)$ is not $V_u/\Gamma\cap\overline{\mathcal{M}}_{0,k+\ell}(K,A,Y)$. This suggests that, the obstruction data of a Kuranishi structure on $\overline{\mathcal{M}}_{0,k+\ell}(K,A,Y)$ must contain location information of the last $\ell$ marked points.

Let $N_{Y,\epsilon}$\label{sym:NYepsilon} be the $\epsilon$-tubular neighborhood of $Y\subset M$, $N_MY$ be the normal bundle of $Y\subset M$ and $\pi_{N_MY}:N_MY\rightarrow Y$ be the {\bf projection}. $N_MY\rightarrow Y$ is a complex line bundle. There is a diffeomorphism
\[\beth_\epsilon:N_{Y,\epsilon}\rightarrow N_MY,\]
\label{Fix:M2:34}which sends $N_{Y,\epsilon}$ to a neighborhood of the zero section of $N_MY$.

By the definition of $\mathcal{J}_{\ell+1}^{*}(M,Y;J,\theta_1)$, \label{Fix:M2:35}$Y$ is $K$-complex. For
\[[(u,{\bf z})]\in \overline{\mathcal{M}}_{0,k+\ell}(K,A,Y)\]
with the underlying curve $\Sigma_u$,  according to Proposition 3.17 and 3.18, $u$ maps no nonconstant component of $\Sigma_u$ into $Y$. Moreover, $u$ intersects $Y$ with isolated intersections. Take a neighborhood $W_u\subset Y$ around $u(\Sigma_u)\cap Y$, such that each component of $W_u$ is diffeomorphic\label{Fix:M:6:4} to the unit ball of $\mathbb{R}^{2n-2}$. Then with a connection $\nabla$ on $M$ satisfying $\nabla J=0$, there is a map $\id_u$ identifying the  fiber of the line bundle $N_MY|_{W_u}\rightarrow W_u$ with $\mathbb{C}$.

\begin{figure}[htbp]
  \centering
  \includegraphics[width=0.8\textwidth]{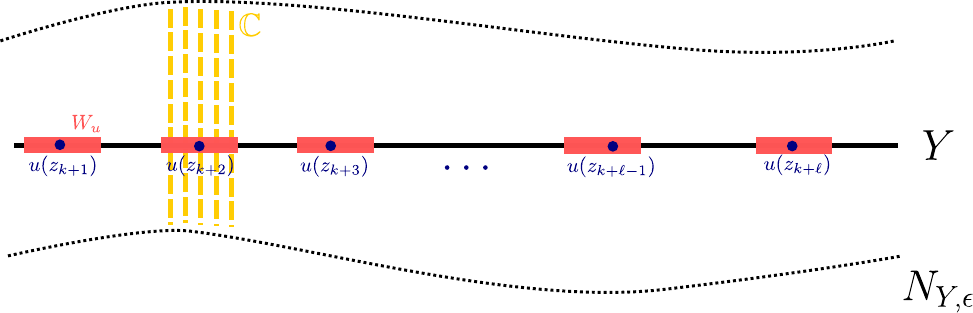}
  \caption{\small These short, thick, horizontal line segments denote components of $W_u$. The vertical lines visualizes fibers of the normal bundle $N_MY\to Y$ on a component of $W_u$.}
  \label{fig:tubularNbhd}\label{Fix:D:1:5}
\end{figure}

\begin{definition}\label{Def:5.1}
  For $u\in \overline{\mathcal{M}}_{0,k+\ell}(K,A,Y)$, assume that $V_u$ is sufficiently small such that for every $v\in V_u$ and $i=k+1,\ldots,k+\ell$, we have $\ev_i(v)\in N_{Y,\epsilon}$ and $\pi_{N_MY}\circ\beth_\epsilon\circ\ev_i(V_u)\subset W_u$. Define maps
  \[\ev_{\mathbb{C},u,i}: V_u\rightarrow\mathbb{C},\ i=k+1,\ldots,k+\ell\]
  by
  \[\ev_{\mathbb{C},u,i}\coloneqq \id_u\circ \ev_i,\ i=k+1,\ldots,k+\ell.\]
  And denote by\label{sym:evcu}
  \[\ev_{\mathbb{C},u}\coloneqq \prod_{i=k+1}^{k+\ell}\ev_{\mathbb{C},u,i}.\]
\end{definition}

\begin{remark}\label{RMK:5.2}
With this definition, we have
\[(s_u\times \ev_{\mathbb{C},u})^{-1}(0)/\Gamma_u=V_u/\Gamma_u\cap \overline{\mathcal{M}}_{0,k+\ell}(K,A,Y).\]
Since the underlying curve $\Sigma_u$ of $u$ is stable, $\Gamma_u$ is a trivial group.
\end{remark}

\begin{theorem}\label{Thm:5.12}
  There exists an oriented Kuranishi structure
  \[\{(V_u,E_{K,u}\times\mathbb{C}^\ell,\Gamma_u,\psi_u,s_u\times \ev_{\mathbb{C},u})\}_{u\in\overline{\mathcal{M}}_{0,k+\ell}(K,A,Y)}\]
  over the space $\overline{\mathcal{M}}_{0,k+\ell}(K,A,Y)$, where $\Gamma_u=\{0\}$ by Remark \ref{RMK:5.2}, and
  \[\psi_u:(s_u\times \ev_{\mathbb{C},u})^{-1}(0)\rightarrow V_u\cap \overline{\mathcal{M}}_{0,k+\ell}(K,A,Y)\]
  is the identity map.
\end{theorem}

\begin{proof}
  Let $(V_u,E_{K,u}\times\mathbb{C}^\ell,\Gamma_u,\psi_u,s_u\times \ev_{\mathbb{C},u})$ and $(V_{u^{'}},E_{K,u^{'}}\times\mathbb{C}^\ell,\Gamma_{u^{'}},\psi_{u^{'}},s_{u^{'}}\times \ev_{\mathbb{C},u^{'}})$ be two Kuranishi neighborhoods, where $u,u^{'}\in \overline{\mathcal{M}}_{0,k+\ell}(K,A,Y)$ and
  \[u^{'}\in (s_u\times \ev_{\mathbb{C},u})^{-1}(0).\]
  The coordinate change $\Phi_{uu^{'}}=(\hat{\phi}_{uu^{'}},\phi_{uu^{'}},h_{uu^{'}})$ is defined as
  \begin{enumerate}
    \item $h_{uu^{'}}:\Gamma_{u^{'}}\rightarrow \Gamma_u$ is the isomorphism between two trivial groups.
    \item $\phi_{uu^{'}}:V_{u^{'}}\cap V_u\rightarrow V_u$ is a smooth embedding as Fukaya's original construction in \cite{FK}.
    \item The map
          \[\hat{\phi}_{uu^{'}}:\mathcal{E}_{K,u^{'}}\times\mathbb{C}^\ell\rightarrow \mathcal{E}_{K,u}\times\mathbb{C}^\ell\]
          has two parts. The part from $\mathcal{E}_{K,u^{'}}$ to $\mathcal{E}_{K,u}$ is defined as in \cite{FK}. The other part, mapping $\mathbb{C}^\ell$ to $\mathbb{C}^\ell$, is given by the coordinate change on the normal bundle $N_Y$ induced by the connection $\nabla$ over $M$. Therefore
          \[\hat{\phi}_{uu^{'}}\circ (s_{u^{'}}\times \ev_{\mathbb{C},u^{'}})=(s_u\times \ev_{\mathbb{C},u})\circ\phi_{uu^{'}}.\]
  \end{enumerate}
  Now we verify that the coordinate change satisfies the following conditions:
  \begin{enumerate}
    \item $\dim V_u-\mathrm{rank}(E_u\times \mathbb{C}^\ell)=\dim V_u-\mathrm{rank}E_u-\ell$ is therefore independent of $u$.
    \item The compatible conditions regarding the actions of $\Gamma_u$ and $\Gamma_{u^{'}}$ are true, since $\Gamma_u$ and $\Gamma_{u^{'}}$ are trivial groups.
  \end{enumerate}

  This Kuranishi structure $\{(V_u,E_{K,u}\times\mathbb{C}^\ell,\Gamma_u,\psi_u,s_u\times \ev_{\mathbb{C},u})\}$ is oriented, \label{Fix:M2:36}since the normal bundle $N_{V_u}V_{u^{'}}$ of $V_{u^{'}}$ in $V_u$ satisfies
  \[N_{V_u}V_{u^{'}}\cong \frac{\ker(\pi_{K,u}\circ D_u\dbar_K)}{\ker(\pi_{K,u^{'}}\circ D_{u^{'}}\dbar_K)},\]
  where $\pi_{J,u}$ denotes the projection
  \[\pi_{J,u}:W^{r-1,p}(\Omega^{0,1}(\Sigma_u,u^{*}TM))\rightarrow \frac{W^{r-1,p}(\Omega^{0,1}(\Sigma_u,u^{*}TM))}{E_{J,u}}.\]
  The quotient vector bundle
  \[\frac{\mathcal{E}_{K,u}\times \mathbb{C}^\ell}{\mathcal{E}_{K,u^{'}}\times \mathbb{C}^\ell}\cong\frac{\mathcal{E}_{K,u}}{\mathcal{E}_{K,u^{'}}}\]
  are only related to $\dbar_K$. The right-hand side of the above equation is the case proved by Fukaya and Ono in \cite{FK}.
\end{proof}

\begin{theorem}\label{Thm:5.15}
  For $K_0,K_1\in\mathcal{J}_{\ell+1}^{*}(M,Y;J,\theta_1)$ in the same connected component, there is a path
  \[K_t:[0,1]\rightarrow \mathcal{J}_{\ell+1}^{*}(M,Y;J,\theta_1)\]
  connecting $K_0$ and $K_1$ such that, one can construct an oriented Kuranishi structure with boundary over the space $\mathcal{W}_{k+\ell}(A,\{K_t\}_t,Y)$, where
  \[\mathcal{W}_{k+\ell}(A,\{K_t\}_t,Y)\coloneqq \{(t,u)|t\in[0,1],u\in\overline{\mathcal{M}}_{0,k+\ell}(K_t,A,Y)\}.\]
  Therefore $\overline{\mathcal{M}}_{0,k+\ell}(K_0,A,Y)$ and $\overline{\mathcal{M}}_{0,k+\ell}(K_1,A,Y)$ are cobordant.
\end{theorem}

\begin{proof}
  Take a small constant $0<\varepsilon<1$, and a path $\{K_t\}\subset \mathcal{J}_{\ell+1}^{*}(M,Y;J,\theta_1)$ connecting $K_0$ and $K_1$, such that
  \[K_t=K_0,\text{ if }t\in[0,\varepsilon],\quad K_t=K_1\text{ if }t\in [1-\varepsilon,1].\]

  According to Theorem \ref{Thm:5.12}, our oriented Kuranishi structure construction applied to the space $\overline{\mathcal{M}}_{0,k+\ell}(K_t,A,Y)$ for every $t\in [0,1]$. Therefore the proof of this theorem then goes exactly the same as Theorem 17.11 in \cite{FK}. For the convenience of readers, we repeat how the obstruction data is constructed here.

  For each $(u,t)$ with $u\in\overline{\mathcal{M}}_{0,k+\ell}(K_t,A,Y)$ and $t\in [0,\varepsilon]\cup[1-\varepsilon,1]$, take the obstruction data $E^{'}_{K_t,u}\times \mathbb{C}^\ell=E^{'}_{K_0,u}\times \mathbb{C}^\ell$ which is constant for $t\in [0,\varepsilon]$ and $E^{'}_{K_t,u}\times \mathbb{C}^\ell=E^{'}_{K_1,u}\times \mathbb{C}^\ell$ for $t\in [1-\varepsilon,1]$.

  For each $(u,t)$ with $t\in (\varepsilon,1-\varepsilon)$, also take \label{Fix:M2:37}an obstruction datum $E^{'}_{K_t,u}\times \mathbb{C}^\ell$. Since $\dbar_{K_t}\times \ev_{\mathbb{C},u}$ smoothly depends on $(u,t)$, there is a small neighborhood
  \[\mathfrak{V}_u\subset \mathcal{W}_{k+\ell}(A,\{K_t\}_t,Y)\]
  for every $u\in \mathcal{W}_{k+\ell}(A,\{K_t\}_t,Y)$, such that $E^{'}_{K_t,u}\times \mathbb{C}^\ell$ taken to $\hat{u}\in \overline{\mathfrak{V}}_u$ via the exponential map is still a valid obstruction data for $\hat{u}$. For $t\in (\varepsilon,1-\varepsilon)$ and $u\in \overline{\mathcal{M}}_{0,k+\ell}(K_t,A,Y)$, take $\mathfrak{V}_u$ small enough such that
  \[\overline{\mathfrak{V}}_u\cap \overline{\mathcal{M}}_{0,k+\ell}(K_t,A,Y)=\emptyset,\ t=0,1.\]

  Now take a finite cover $\{\mathfrak{V}_{u_i}\}_i$ of $\mathcal{W}_{k+\ell}(A,\{K_t\}_t,Y)$, then for each $(u,t)$ consider all $(u_{i_l},t_{i_l})$ with $u\in \overline{\mathfrak{V}}_{u_{i_l}}$ and define
  \[E_{K_t,u}\coloneqq \bigoplus_{l}\exp_{u_{i_l},u} (E^{'}_{K_{t_{i_l}},u_{i_l}}),\]
  where $\exp_{u_{i_l},u}$ is the exponential map from $u_{i_l}$ to $u$. $E_{K_t,u}\times \mathbb{C}^\ell$ for each $u\in \mathcal{W}_{k+\ell}(A,\{K_t\}_t,Y)$ is the desired obstruction data.
\end{proof}

\subsection{Pseudocycle and multisection on $\mathcal{M}_{0,k+\ell}(K,A,Y)$}\label{Sec:5.2}
Recall that ${\bf ev}_k\coloneqq \prod_{i=1}^k\ev_i$. According to Theorem \ref{Thm:3.25}, there exists a Baire set 
\[\mathcal{J}_{\ell+1}^{reg*}\subset \mathcal{J}_{\ell+1}^{*}(M,Y;J,\theta_1)\]
such that for $K\in \mathcal{J}_{\ell+1}^{reg*}$, the evaluation map
\[{\bf ev}_k:\mathcal{M}_{0,k+\ell}(K,A,Y)\rightarrow M^k\]
represents a pseudocycle.

On the other hand, the zeros of a generic multisection on the oriented Kuranishi structure $\{(V_u,E_{K,u}\times\mathbb{C}^\ell,\Gamma_u,\psi_u,s_u\times \ev_{\mathbb{C},u})\}$ defined in Theorem \ref{Thm:5.12}, with the evaluation map ${\bf ev}_k$, provides a virtual fundamental class. We point out that, since $\Gamma_u$ is trivial, this class lies in $H_{*}(M^k;\mathbb{Z})$.

In this section, we compare the pseudocycle and the virtual fundamental class described above, and aim to construct equivalence between them.

\begin{lemma}\label{Lemma:5.4}
  For $K\in \mathcal{J}_{\ell+1}^{*}(M,Y;J,\theta_1)$, take the oriented Kuranishi structure as in Theorem \ref{Thm:5.12}
  \[\{(V_u,E_{K,u}\times\mathbb{C}^\ell,\Gamma_u,\psi_u,s_u\times\ev_{\mathbb{C},u})\}_{u\in\overline{\mathcal{M}}_{0,k+\ell}(K,A,Y)}.\]
  Let $s^{'}=\{s_u^{'}\}$ be a multisection, where each $s_u^{'}$ is smooth. \label{Fix:M2:38}Assume that for a compact set $\mathfrak{C}\subset s^{'-1}(0)$, $s_u^{'}$ is transverse to $0$ in a neighborhood of $\mathfrak{C}\cap V_u$. Then there exists a multisection $h=\{h_u\}$ such that
  \begin{enumerate}
    \item $h$ is transverse to $0$,
    \item $h_u=s_u^{'}$ on a neighborhood of  $V_u\cap \mathfrak{C}$,
    \item $h_u$ can be chosen arbitrarily close to $s_u^{'}$ in $C^\infty$ topology on $V_u$.
  \end{enumerate}
\end{lemma}

\begin{proof}
  Take a finite cover $\{V_{u_i}\}_{i=1}^{\mathfrak{l}}$ of $\overline{\mathcal{M}}_{0,k+\ell}(K,A,Y)$, and order it such that
  \[\mathrm{rank} E_{K,u_i}\leq \mathrm{rank} E_{K,u_{i+1}}.\]
  We construct $h$ by induction with respect to $i$.

  {\bf\small Induction step I.} For $i=1$, let $N_{u_1}$ be a \label{Fix:M2:39}open neighborhood of $\mathfrak{C}\cap V_{u_1}$ such that $\overline{N}_{u_1}\subset V_{u_1}$, and $s_{u_1}$ is transverse to $0$ in a neighborhood of $\overline{N}_{u_1}$. According to \cite[Lem.~3.14]{FK}, there exists a multisection $h_{u_1}$ on $V_{u_1}$ such that
  \begin{enumerate}
    \item $h_{u_1}$ is transverse to $0$,
    \item $h_{u_1}$ is sufficiently close to $s_{u_1}^{'}$ on $V_{u_1}$ in $C^\infty$ topology,
    \item $h_{u_1}=s_{u_1}$ on $\overline{N}_{u_1}$.
  \end{enumerate}

  {\bf\small Induction step II.} Assume that for a $1\leq j\leq \mathfrak{l}$, $h_{u_i}$ and the open neighborhood $N_{u_i}$ are constructed for all $1\leq i\leq j$. Then on $V_{u_{j+1}}$, let $N_{u_{j+1}}$ be an open neighborhood of $\mathfrak{C}\cap V_{u_{j+1}}$ such that $\overline{N}_{u_{j+1}}\subset V_{u_{j+1}}$, and $s_{u_{j+1}}$ is transverse to $0$ on a neighborhood of $\overline{N}_{u_{j+1}}$. Moreover, we ask that
  \[\overline{N}_{u_{j+1}}\cap V_{u_i}\subset \overline{N}_{u_i},\ 1\leq i\leq I.\]
  As in the proof of Theorem 6.4 in \cite{FK},
  \begin{enumerate}
    \item Take a tubular neighborhood $N_{V_{u_{j+1}}}V_{u_i}\subset V_{u_{j+1}}$ of $V_{u_i}\cap V_{u_{j+1}}$ for each $i\leq j$. Extend $h_{u_i}|_{V_{u_i}\cap V_{u_{j+1}}}$ to $N_{V_{u_{j+1}}}V_{u_i}$ by
          \[h_{u_{j+1}u_i}(v)\coloneqq  \Phi_{u_{j+1}u_i}(\rho)\oplus h_{u_i}(\tau).\]
    where $v\in N_{V_{u_{j+1}}}V_{u_i}$, and $v=(\tau,\rho)$ parametrizes $N_{V_{u_{j+1}}}V_{u_i}$ with $\tau\in V_{u_i}\cap V_{u_{j+1}}$ and $\rho$ denoting the normal direction of $N_{V_{u_{j+1}}}V_{u_i}$, map $\Phi_{u_{j+1}u_i}$ is the isomorphism (see \cite[Def.~5.9]{FK})
          \[\Phi_{u_{j+1}u_i}:N_{V_{u_{j+1}}}V_{u_i}\rightarrow E_{K,u_{j+1}}/E_{K,u_i}.\]
    \item According to \cite[Lem.~3.14]{FK}, there exists a multisection $h_{u_{j+1}}$ on $V_{u_{j+1}}$ such that
    \begin{enumerate}
      \item $h_{u_{j+1}}$ is transverse to $0$,
      \item $h_{u_{j+1}}$ is sufficiently close to $s_{u_{j+1}}^{'}$ on $V_{u_{j+1}}$ in $C^\infty$ topology,
      \item $h_{u_{j+1}}=s_{u_{j+1}}$ on $\overline{N}_{u_{j+1}}$.
      \item $h_{u_{j+1}}=h_{u_{j+1}u_i}$ on $N_{V_{u_{j+1}}}V_{u_i}$ for each $i\leq j$.
  \end{enumerate}
  \end{enumerate}
  This completes the construction.
\end{proof}

Let $T_r$ be the set consisting of all $\ell$-stable $(k+\ell)$-labeled trees $T$ with more than one node, and let
\[\mathcal{M}_{T_r}(K,A,Y)\coloneqq \coprod_{T\in T_r}\mathcal{M}_{0,T}(K,A,Y).\]

\begin{proposition}\label{Prop:5.5}
   Take $K\in \mathcal{J}_{\ell+1}^{reg*}$. Notice that ${\bf ev}_k(\mathcal{M}_{T_r}(K,A,Y))$ is a compact subset of $M^k$. For any open neighborhood $O$ of ${\bf ev}_k(\mathcal{M}_{T_r}(K,A,Y))$ in $M^k$, there exists a section $h=\{h_u\}$ of the obstruction bundle of the Kuranishi structure over  $\overline{\mathcal{M}}_{0,k+\ell}(K,A,Y)$, such that
  \begin{enumerate}
    \item $h$ is transverse to $0$,
    \item denote by ``$\bigtriangleup$'' the symmetric difference of sets, then
          \[\bigcup_{u\in \overline{\mathcal{M}}_{0,k+\ell}(K,A,Y)}\left({\bf ev}_k (h_u^{-1}(0))\bigtriangleup {\bf ev}_k ((s_u\times ev_{\mathbb{C},u})^{-1}(0))\right)\subset O.\]
  \end{enumerate}
\end{proposition}

\begin{proof}
  Take an open neighborhood $O$ in $M^k$ of ${\bf ev}_k(\mathcal{M}_{T_r}(K,A,Y))$ and a compact subset $\mathfrak{C}\subset \mathcal{M}_{0,k+\ell}(K,A,Y)$, such that the closure of ${\bf ev}_k\left(\overline{\mathcal{M}}_{0,k+\ell}(K,A,Y)\backslash \mathfrak{C}\right)$ in $M^k$ satisfies
  \[\overline{{\bf ev}_k\left(\overline{\mathcal{M}}_{0,k+\ell}(K,A,Y)\backslash \mathfrak{C}\right)}\subset O.\]
  By Lemma \ref{Lemma:5.4}, there exists a section $h=\{h_u\}$ such that
  \begin{enumerate}
    \item $h$ is transverse to $0$,
    \item $h_u=s_u\times ev_{\mathbb{C},u}$ in a tubular neighborhood of $V_u\cap \mathfrak{C}$, for all $u$,
    \item $h_u$ is sufficiently close to $s_u\times ev_{\mathbb{C},u}$ such that the zeros of $h_u$ are close enough to $V_u\cap \overline{\mathcal{M}}_{0,k+\ell}(K,A,Y)$ to ensure that
          \[{\bf ev}_k (v)\in O\text{ for all }v\in h_u^{-1}(0)\backslash\mathfrak{C}.\]
  \end{enumerate}
  With such a $h=\{h_u\}$, we have
  \begin{align*}
    &{\bf ev}_k (h_u^{-1}(0))\bigtriangleup{\bf ev}_k ((s_u\times ev_{\mathbb{C},u})^{-1}(0))\\
    \subset& {\bf ev}_k (h_u^{-1}(0)\backslash\mathfrak{C})\cup{\bf ev}_k ((s_u\times ev_{\mathbb{C},u})^{-1}(0)\backslash\mathfrak{C})\\
    \subset&O.
  \end{align*}
\end{proof}

Notably, if we take the homology class defined by the pesudocycle ${\bf ev}_k$ as described by A.Zinger in \cite{Zinger}, then we have the following result:

\begin{theorem}\label{Thm:5.14}
  For $K\in \mathcal{J}^{reg*}_{\ell+1}$, using the homology class induced by a pesudocycle as described in Theorem \ref{Thm:3.3.1}, then the pseudocycle
  \[{\bf ev}_k :\mathcal{M}_{0,k+\ell}(K,A,Y)\rightarrow M^k\]
  induces  the same homology class in $H_d(M^k;\mathbb{Z})$ as the virtual fundamental class.
\end{theorem}

\begin{proof}
We choose an open neighborhood $O\subset M^k$ of ${\bf ev}_k(\mathcal{M}_{T_r}(K,A,Y))$ as in Proposition \ref{Prop:3.3.1} and Equation (\ref{eq:3.3.3}), and a compact set $\mathfrak{C}\subset \mathcal{M}_{0,k+\ell}(J_0,A,Y)$ as in the proof of Proposition \ref{Prop:5.5}. The evaluation map ${\bf ev}_k:\mathfrak{C}\rightarrow M^k$ defines a homology class 
\[[{\bf ev}_k]\in H_{2(n-3+k+c_1(A))}(M^k,O;\mathbb{Z})\cong H_{2(n-3+k+c_1(A))}(M^k;\mathbb{Z}).\]
The resulting class $[{\bf ev}_k]$ in $H_{2(n-3+k+c_1(A))}(M^k;\mathbb{Z})$ is independent of the choice of the neighborhood $O$ and the compact set $\mathfrak{C}$.

Now according to Proposition \ref{Prop:5.5}, there exists a section $h=\{h_u\}$, such that
\begin{enumerate}
  \item $h$ is transverse to $0$,
  \item $h_u=s_u\times ev_{\mathbb{C},u}$ in a tubular neighborhood of $V_u\cap \mathfrak{C}$, for all $u$,
  \item denote by ``$\bigtriangleup$'' the symmetric difference of sets, then
          \[\bigcup_{u\in \overline{\mathcal{M}}_{0,k+\ell}(K,A,Y)}\left({\bf ev}_k (h_u^{-1}(0))\bigtriangleup {\bf ev}_k ((s_u\times ev_{\mathbb{C},u})^{-1}(0))\right)\subset O.\]
\end{enumerate}
Therefore the following calculation of relative homology classes
\begin{equation*}
  [{\bf ev}_k:\mathfrak{C}\to M^k]-{\bf ev}_{k*}([h^{-1}(0)])=0\in H_{2(n-3+k+c_1(A))}(M^k,O;\mathbb{Z})
\end{equation*}
implies
\[[{\bf ev}_k:\mathfrak{C}\to M^k]={\bf ev}_{k*}([h^{-1}(0)])\text{ in }H_{2(n-3+k+c_1(A))}(M^k;\mathbb{Z}).\]
This completes the proof.
\end{proof}

\subsection{Second Kuranishi structure on $\overline{\mathcal{M}}_{0,k+\ell}(J_0,A)$ with modified obstruction data}\label{section:5.1}\label{Fix:M2:43}
According to Section \ref{Sec:3}, there exists {\bf domain independent} 
\[J_0\in B^{*}\subset \mathcal{J}_{\ell+1}^{*}(M,Y;J,\theta_1)\]
such that $Y$ is $J_0$-holomorphic and for sufficiently large $D$, the following conditions are satisfied (see Definition \ref{Def:3.20}, Proposition \ref{Prop:3.16}).
\begin{enumerate}
    \item All $J_0$-holomorphic spheres of energy at most $E_\ell$ contained in $Y$ are constant.
    \item Every nonconstant $J_0$-holomorphic sphere of energy at most $E_\ell$ in $X$ intersects $Y$ in at least $3$ distinct points in the domain,
\end{enumerate}
where $E_\ell=\ell/D$, $\ell=D\omega(A)$.

Let $\{(V_u,E_{J_0,u},\Gamma_u,\psi_u,s_u)\}$ be the Kuranishi structure on $\overline{\mathcal{M}}_{0,k}(J_0,A)$, $\Sigma_u$ the underlying curve of $u$, and $u^{-1}(Y)$ the preimage of $Y$ on $\Sigma_u$. In this section, we will modify the obstruction data $E_{J_0,u}$ on a $\epsilon$-neighborhood of $u^{-1}(Y)$, then obtain new obstruction data $E_{J_0,u,\epsilon}$ and another Kuranishi structure on $\overline{\mathcal{M}}_{0,k}(J_0,A)$.

Firstly we prove that, if one modifies $E_{J_0,u}$ on a sufficiently small neighborhood on $\Sigma_u$, the result will still be a complement of $\ker(\pi_{J_0,u}\circ D_u\bar{\partial}_{J_0})\subset W^{r,p}(\Sigma_u,M)$, where $\pi_{J_0,u}$ is the projection map
\[\pi_{J_0,u}:W^{r-1,p}(\Omega^{0,1}(\Sigma_u,u^{*}TM))\rightarrow \frac{W^{r-1,p}(\Omega^{0,1}(\Sigma_u,u^{*}TM))}{E_{J_0,u}}.\]
\begin{lemma}[F.Riesz \cite{Riesz}]
  \label{RieszLemma}
  If $X$ is a proper closed subspace of a Banach space $X_0$, then for all $0<\epsilon<1$, there exists $y\in X_0$ such that $\|y\|=1$ and $\|y-x\|\geq 1-\epsilon$ $(\forall x\in X)$.
\end{lemma}

\begin{lemma}\label{lemma3.6}
  Let $X,Y$ be two Banach spaces, and $\zeta$ be a Fredholm operator from $X$ to $Y$. Let $E$ be a finite dimensional subspace of $Y$, and $\mathcal{S}_E$ be the unit sphere of $E$. Suppose that $\Img\zeta\oplus E=Y$ and $E\cap \Img\zeta=\{0\}$. Then there exists a constant $c>0$ depending on $E$ and $\zeta$ such that for all finite dimensional subspace $E^{'}\subset Y$ with
  \[\max_{x\in\mathcal{S}_E}\dist(x,E^{'})\leq c,\]
  we have $\Img\zeta\oplus E^{'}=Y$, where $\dist(x,E^{'})$ denotes the distance between $x$ and $E^{'}$.
\end{lemma}

\begin{proof}
  Assume that $\Img\zeta\oplus E^{'}\neq Y$, then by Lemma \ref{RieszLemma} for any $\epsilon>0$ there exists a $y\in Y$, $\|y\|=1$ such that
  \[\dist(y,\Img\zeta\oplus E^{'})\geq 1-\epsilon.\]
  Since $\Img\zeta\oplus E=Y$, we can decompose $y$ as $y=l+e$, where $l\in\Img\zeta$ and $e\in E$. Therefore for any $w\in \Img\zeta\oplus E^{'}$
  \begin{equation*}
    1-\epsilon\leq \|y-w\|=\|l+e-w\|=\|e-(l-w)\|.
  \end{equation*}
  $w$ is an arbitrary element in $\Img\zeta\oplus E^{'}$, thus $l-w$ is also arbitrary in $\Img\zeta\oplus E^{'}$. We conclude that
  \[\dist(e,\Img\zeta\oplus E^{'})\geq 1-\epsilon.\]

  On the other hand
  \[1=\|y\|=\|l+e\|\geq\dist(e,\Img\zeta)=\|e\|\dist(\frac{e}{\|e\|},\Img\zeta)\geq\|e\|\dist(\mathcal{S}_E,\Img\zeta).\]
  which implies
  \[\|e\|\leq 1/\dist(\mathcal{S}_E,\Img\zeta). \]
  Therefore
  \begin{align*}
    \max_{x\in\mathcal{S}_E}\dist(x,E^{'})\geq \dist(\frac{e}{\|e\|},E^{'})=\frac{\dist(e,E^{'})}{\|e\|} \geq (1-\epsilon)\dist(\mathcal{S}_E,\Img\zeta).
  \end{align*}
  Since $\epsilon$ is arbitrary, we conclude that for any $E^{'}$ satisfying
  \[\max_{x\in\mathcal{S}_E}\dist(x,E^{'})<\dist(\mathcal{S}_E,\Img\zeta),\]
  \label{Fix:M2:44}we have $E^{'}\oplus \Img \zeta=Y$. Taking $c=\dist(\mathcal{S}_E,\Img\zeta)$ completes the proof.
\end{proof}

\begin{corollary}\label{corollary3.3}
  The condition $E\cap\Img\zeta=\{0\}$ is not necessary in Lemma \ref{lemma3.6}.
\end{corollary}

\begin{proof}
  Assume that $E\oplus \Img\zeta=Y$, then there exist a subspace $\hat{E}\subset E$ such that $\hat{E}\oplus \Img\zeta=Y$ and $\hat{E}\cap \Img\zeta=\{0\}$. Assume that
  \[\hat{c}\coloneqq \max_{\substack{\hat{E}\subset E,\\\hat{E}\oplus \Img\zeta=Y,\\\hat{E}\cap \Img\zeta=\{0\}}}\dist(\mathcal{S}_{\hat{E}},\Img\zeta).\]
  Then $\hat{c}$ depends only on $\zeta, E$, and there exists at least one $\hat{E}$ such that 
  \[\dist(\mathcal{S}_{\hat{E}},\Img\zeta)\geq\frac{\hat{c}}{2}.\]
  Apply lemma \ref{lemma3.6} to this $\hat{E}$ we complete the proof for $c=\frac{\hat{c}}{2}$.
\end{proof}

\begin{lemma}\label{lemma3.7}\label{Fix:M2:45}
  For a $k$-labeled tree $T$, let $\Sigma_T$ be the genus zero curve modeled on $T$, $\dbar_{J_0}:W^{1,p}(\Sigma_T,M)\rightarrow M$ be the Cauchy-Riemann operator with $p>2$. Let $D_u\dbar_{J_0}$ be the linearization of $\dbar_{J_0}$ at $u\in W^{1,p}(\Sigma_T,M)$. Denote by $N_{Y,2\epsilon}\subset M$ a $2\epsilon$-tubular neighborhood of $Y$ for some $\epsilon>0$.

  For $u\in W^{1,p}(\Sigma_T,M)$, denote by $\chi_\epsilon$ a cutoff function supported on $u^{-1}(N_{Y,2\epsilon})$ such that for any
  \[z=u^{-1}(y),y\in Y,\]
  $\chi_\epsilon\equiv 1$ on $u^{-1}(B_y(\epsilon))$. Suppose that $E\subset W^{0,p}(\Omega^{0,1}(\Sigma_T,u^{*}TM))$ is a finite dimensional subspace such that
  \[\Img D_u\bar{\partial}_{J_0}\oplus E=W^{0,p}(\Omega^{0,1}(\Sigma_T,u^{*}TM)),\]
  then for $\epsilon$ small enough the space
  \[ E_\epsilon\coloneqq\{(1-\chi_\epsilon)f:f\in E\} \]
  also satisfies
\[\Img D_u\bar{\partial}_{J_0}\oplus E_\epsilon=W^{0,p}(\Omega^{0,1}(\Sigma_T,u^{*}TM)).\]
\end{lemma}

\begin{proof}
  Assume that $z\in u^{-1}(Y)$ belongs to a non-ghost component $S_\alpha\subset \Sigma_T$. According to Lemma \ref{lemma3.6} and Corollary \ref{corollary3.3}, we only need to prove that 
  \[\max_{f\in\mathcal{S}_E}\dist(f,E_\epsilon)\]
  can be arbitrarily small when $\epsilon$ is small enough.

  Since $u|_{S_\alpha}$ is not constant and
  \[J_0\in B^{*}\subset \mathcal{J}_{\ell+1}^{*}(M,Y;J,\theta_1),\]
  we have that $u(S_\alpha)$ intersects $Y$ transversely (see Proposition \ref{Prop:3.16}). Thus the component of $u^{-1}(N_{Y,2\epsilon})$ containing $z$ is inside a ball $B_z(2\epsilon)\subset\Sigma_T$ of radius $2\epsilon$ (with a suitable metric). Therefore
  \begin{align*}
    \dist(f,E_\epsilon)\leq &\dist(f,(1-\chi_\epsilon)f) \\
    =&\|f-(1-\chi_\epsilon)f\|_{L^p} \\
    =&\|\chi_\epsilon f\|_{L^p} \rightarrow 0\quad \text{ as } \epsilon\rightarrow 0.
  \end{align*}
  This is true for all $f\in\mathcal{S}_E$.

  If there is $z\in u^{-1}(Y)$ that belongs to a ghost sphere $S_\alpha$, let $U_z\subset \Sigma_T$ be the connected component of $u^{-1}(N_{Y,2\epsilon})$ containing $z$, then $S_\alpha\subset U_z$ and $(1-\chi_\epsilon)f$ vanishes on $S_\alpha$. Moreover if $S_\beta$ is a non-ghost component connected to $S_\alpha$ at a nodal point $z_{\alpha\beta}$, then $(1-\chi_\epsilon)f$ vanishes on $U_z\cap S_\beta$, which is a neighborhood of $z_{\alpha\beta}$ inside a ball of radius $2\epsilon$.

  Even though elements of $E_\epsilon$ vanish on ghost components, we still have
  \[\Img D_u\bar{\partial}_{J_0}\oplus E_\epsilon=L^p(\Omega^{0,1}(\Sigma_T,u^{*}TM))\]
  because $D_u\bar{\partial}_{J_0}$ is already transverse at constant maps.
\end{proof}

By the elliptic regularity (see \cite[App.~B]{MS}) of the dual operator of $D_u\bar{\partial}_{J_0}$, one may assume that $E_\epsilon$ consists of smooth sections, and therefore
\[E_\epsilon\subset W^{r-1,p}(\Omega^{0,1}(\Sigma_T,u^{*}TM))\]
for all $p>1$ and positive integer $r$ with $r-2/p>1$.
\begin{proposition}\label{Fix:M2:46}
  For a $k$-labeled tree $T$, let $\Sigma_T$ be the genus zero curve modeled on $T$, and $\dbar_{J_0}:W^{r,p}(\Sigma_T,M)\rightarrow M$ be the Cauchy-Riemann operator, where $r$ is a positive integer, $p>1$ and $r-2/p>1$. Let $D_u\dbar_{J_0}$ be the linearization of $\dbar_{J_0}$ at $u\in W^{r,p}(\Sigma_T,M)$.

  Then for $\epsilon>0$ sufficiently small, the vector space $E_\epsilon$ defined in Lemma \ref{lemma3.7} satisfies
  \[\Img D_u\bar{\partial}_{J_0}\oplus E_\epsilon=W^{r-1,p}(\Omega^{0,1}(\Sigma_T,u^{*}TM))\]
  for any integer $r> 0$.
\end{proposition}

\begin{proof}
  Firstly, assume that $\Img D_u\bar{\partial}_{J_0}\cap E_\epsilon=\{0\}$. We prove this by induction with respect to $r$. The case of $r=1$ is proven by Lemma \ref{lemma3.7}.

  Assume it holds for all $r\leq r_0$. Denote by
  \[\Img_{r_0+1}D_u\bar{\partial}_{J_0}\subset W^{r_0,p}(\Omega^{0,1}(\Sigma_T,u^{*}TM))\]
  the image of $D_u\bar{\partial}_{J_0}:W^{r_0+1,p}(\Sigma_T,M)\rightarrow M$, then by the following commutative diagram
\[\begin{tikzcd}
	W^{r_0+1,p}(\Sigma_T,M) && W^{r_0,p}(\Omega^{0,1}(\Sigma_T,u^{*}TM)) \\
	\\
	W^{r_0,p}(\Sigma_T,M) && W^{r_0-1,p}(\Omega^{0,1}(\Sigma_T,u^{*}TM))
	\arrow["\bar{\partial}_{J_0}", from=1-1, to=1-3]
	\arrow["\iota"', from=1-1, to=3-1]
	\arrow["\bar{\partial}_{J_0}"', from=3-1, to=3-3]
	\arrow["\hat{\iota}", from=1-3, to=3-3]
  \end{tikzcd}\]
where $\iota$, $\hat{\iota}$ are the natural embeddings, we have
\[\hat{\iota}(\Img_{r_0+1}D_u\bar{\partial}_{J_0})\subset \Img_{r_0}D_u\bar{\partial}_{J_0},\]
and therefore by the assumption,
\[\hat{\iota}(\Img_{r_0+1}D_u\bar{\partial}_{J_0})\cap E_\epsilon\subset \Img_{r_0}D_u\bar{\partial}_{J_0}\cap E_\epsilon=\{0\}.\]
Since $\hat{\iota}$ is an embedding, we conclude that $\Img_{r_0+1}D_u\bar{\partial}_{J_0}\cap E_\epsilon=\{0\}$.

Now since the dimension of the cokernel of $D_u\bar{\partial}_{J_0}$ does not depend on $r$, and $E_\epsilon$ is exactly of that dimension, we have
\[\Img D_u\bar{\partial}_{J_0}\oplus E_\epsilon=W^{r_0,p}(\Omega^{0,1}(\Sigma_T,u^{*}TM)).\]
This completes the induction.

If $\Img D_u\bar{\partial}_{J_0}\cap E_\epsilon\neq\{0\}$, we take a subspace $\hat{E}_\epsilon\subset E_\epsilon$ satisfies
\begin{align*}
  \Img D_u\bar{\partial}_{J_0}\oplus \hat{E}_\epsilon=&W^{0,p}(\Omega^{0,1}(\Sigma_T,u^{*}TM))\\
  \Img D_u\bar{\partial}_{J_0}\cap \hat{E}_\epsilon=&\{0\}.
\end{align*}
And repeat the argument above replacing $E_\epsilon$ with $\hat{E}_\epsilon$.
\end{proof}

Back to the construction of the Kuranishi structure. Let 
\[\{(V_u,E_{J_0,u},\Gamma_u,\psi_u,s_u)\}\]
 be the Kuranishi structure on $\overline{\mathcal{M}}_{0,k}(J_0,A)$ as in Section \ref{Sec:4}. Let us modify the obstruction data $E_{J,u}$ as in the proof of Lemma \ref{lemma3.7}, and denote the result by $E_{J_0,u,\epsilon_u}$
\[E_{J_0,u,\epsilon_u}\coloneqq \{(1-\chi_{\epsilon_u})f:f\in E_{J_0,u}\}.\]

\begin{lemma}\label{lemma:5.17}
For $\{E_{J,u,\epsilon_u}\}$ defined above, there exists a universal $\epsilon>0$ for all $u$ to replace $\epsilon_u$.
\end{lemma}

\begin{proof}
  This is a standard compact argument. For each $u\in\overline{\mathcal{M}}_{0,k}(J_0,A)$, there is a neighborhood $\mathfrak{U}_u\subset \overline{\mathcal{M}}_{0,k}(J_0,A)$ of $u$, such that $\epsilon_v$ has a positive lower bound $\epsilon_{\mathfrak{U}_u}$ for $v\in \mathfrak{U}_u$. Since there exists a finite subcover $\{\mathfrak{U}_{u_i}\}_{1\leq i\leq \mathfrak{m}}\subset \{\mathfrak{U}_u\}$, take the smallest $\epsilon_{\mathfrak{U}_{u_i}}$ as our $\epsilon$.
\end{proof}

Now we take $\{E_{J_0,u,\epsilon}\}$\label{sym:Ejwye} as our {\bf new obstruction data} for $\overline{\mathcal{M}}_{0,k}(J_0,A)$.

\begin{remark}
 Denote by $N_{Y,\epsilon}\subset M$ be the $\epsilon$-tubular neighborhood of $Y$. Then $E_{J_0,u,\epsilon}$ is supported on $\Sigma_T\setminus u^{-1}(N_{Y,\epsilon})$. Let $z_i$ for $i=k+1,\ldots,k+\ell$ be the $i$-th marked point on $\Sigma_T$. The preimage $u^{-1}(N_{Y,\epsilon})$
 is clearly $\Gamma_u$-invariant, and therefore $\Gamma_u$ action on $E_{J_0,u,\epsilon}$ is well-defined.
\end{remark}

Repeat the construction in Section \ref{Sec:4}, we conclude the following.
\begin{proposition}\label{proposition3.28}\label{Fix:M2:47}
  For sufficiently small $\epsilon>0$, an oriented Kuranishi structure 
  \[\{(V_u,E_{J_0,u,\epsilon},\Gamma_u,\psi_u,s_u)\}\]
  is defined on $\overline{\mathcal{M}}_{0,k}(J_0,A)$ with respect to the obstruction data $E_{J_0,u,\epsilon}$, where
  \[E_{J_0,u,\epsilon}\coloneqq \{(1-\chi_\epsilon)f:f\in E_{J_0,u}\},\]
  $\chi_\epsilon:u^{-1}(B_y(2\epsilon))\rightarrow [0,1]$ is a cut-off function which is supported in the preimage of the ball $B_y(2\epsilon)$ centered at $y\in u\cap Y$ with radius $2\epsilon$, and $\chi_\epsilon\equiv 1$ on $u^{-1}(B_y(\epsilon))$. Thus elements of $E_{J_0,u,\epsilon}$ vanish in $u^{-1}(N_{Y,\epsilon})$.
\end{proposition}

\subsection{``Pullback'' Kuranishi structure on $\overline{\mathcal{M}}_{0,k+\ell}(J_0,A,Y)$}\label{Sec:PullBack}
Fix a constant $\theta_1$ with $\theta_2<\theta_1<\theta_0$ as in Proposition \ref{Prop:3.16}, and choose a {\bf domain independent} element $J_0\in B^{*}\subset \mathcal{J}_{\ell+1}^{*}(M,Y;J,\theta_1)$, where $B^{*}$ is an $\theta_2$-contractible open subset of $\mathcal{J}^{*}(M,Y;J,\theta_1,E_\ell)$ (see Proposition \ref{Prop:3.16} and Definition \ref{Def:3.20}).

Throughout this section, we denote by $[(u,{\bf z})]$ the elements of $\overline{\mathcal{M}}_{0,k+\ell}(J_0,A,Y)$, and by $[(w,{\bf z})]$ the elements of $\overline{\mathcal{M}}_{0,k}(J_0,A)$ to distinguish them. For convenience, we also use $u$ and $w$ as abbreviations for $[(u,{\bf z})]$ and $[(w,{\bf z})]$ respectively.

\subsubsection{Definition and properties of the forgetful maps}
We aim to compare the virtual fundamental classes represented by ${\bf ev}_k\coloneqq \prod_{i=1}^k\ev_i$ with the Kuranishi structures
\begin{equation}\label{eq:5.19}
\{(V_u,E_{J_0,u}\times\mathbb{C}^\ell,\Gamma_u,\psi_u,s_u\times \ev_{\mathbb{C},u})\}_{u\in\overline{\mathcal{M}}_{0,k+\ell}(J_0,A,Y)}
\end{equation}
defined in Theorem \ref{Thm:5.12}, and
\begin{equation}\label{eq:5.20}
  \{(V_w,E_{J_0,w},\Gamma_w,\psi_w,s_w)\}_{w\in \overline{\mathcal{M}}_{0,k}(J_0,A)}
\end{equation}
defined in Section \ref{Sec:4}. 

Firstly, the virtual fundamental classes do not depend on the choices made to define the Kuranishi structure. Therefore on $\overline{\mathcal{M}}_{0,k}(J_0,A)$, instead of (\ref{eq:5.20}), we shall use the Kuranishi structure 
$\{(V_w,E_{J_0,w,\epsilon},\Gamma_w,\psi_w,s_w)\}$ defined in Proposition \ref{proposition3.28}. Next, the plan is to find a ``pullback'' Kuranishi structure on the moduli space $\overline{\mathcal{M}}_{0,k+\ell}(J_0,A,Y)$ to replace (\ref{eq:5.19}). The "pullback" is defined by the forgetful maps $\varphi_\ell$ and $\underline{\varphi}_\ell$ described below.

\begin{definition}\label{Def:5.20}\label{sym:forgetful}
  Define
  \[\varphi_\ell:\overline{\mathcal{M}}_{0,k+\ell}(J_0,A,Y)\rightarrow \overline{\mathcal{M}}_{0,k}(J_0,A)\]
  to be the {\bf forgetful map of maps}. For $u\in \overline{\mathcal{M}}_{0,k+\ell}(J_0,A,Y)$, $\varphi_\ell$ forgets the last $\ell$ marked points of $u$ and subsequently stabilizes its ghost components, while $u$ remains unchanged on all remaining components.

  Define $\underline{\varphi}_\ell$ to be the {\bf forgetful map of curves}. Let $\Sigma_u$ be the underlying curve of $u\in \overline{\mathcal{M}}_{0,k+\ell}(J_0,A,Y)$. The map $\underline{\varphi}_\ell$ forgets the last $\ell$ marked points of $\Sigma_u$ and subsequently stabilizes its ghost components, resulting in a genus zero curve with $k$ marked points.
\end{definition}

\begin{definition}\label{def:5.21}
  Let $u\in\overline{\mathcal{M}}_{0,k+\ell}(J_0,A,Y)$ and $w= \varphi_\ell(u)$. Then define $\widetilde{E}_{J_0,u,\epsilon}$ as
  \[\widetilde{E}_{J_0,u,\epsilon}\coloneqq \{\underline{\varphi}_\ell^{*}(\kappa)|\kappa\in E_{J_0,w,\epsilon}\},\]
  where $\underline{\varphi}_\ell^{*}$ is the pullback map corresponding to $\underline{\varphi}_\ell$.
\end{definition}

\begin{remark}\label{RMK:5.18}
  Let $\Sigma_u$ be the underlying curve of $u\in \overline{\mathcal{M}}_{0,k+\ell}(J_0,A,Y)$. Suppose that $\Sigma_u$ is modeled on the $\ell$ stable $k+\ell$-labeled tree $T$. Denote by $\{S_\alpha\}_{\alpha\in T}$ the sphere components of $\Sigma_u$.
  \begin{enumerate}
    \item If $S_\alpha$ is not contracted by $\underline{\varphi}_\ell$, then when restricted to $S_\alpha$,
          \[\widetilde{E}_{J_0,u,\epsilon}|_{S_\alpha}\cong E_{J_0,w,\epsilon}|_{\underline{\varphi}_\ell(S_\alpha)}\]
          via the map $\underline{\varphi}_\ell^{*}$.
    \item If $S_\alpha$ is a ghost sphere containing some of the last $\ell$ marked points, then
          \[\widetilde{E}_{J_0,u,\epsilon}|_{S_\alpha}=\{0\}.\]
  \end{enumerate}
  Since elements of $\widetilde{E}_{J_0,u,\epsilon}$ vanish near nodal points of $\Sigma_u$, $\widetilde{E}_{J_0,u,\epsilon}$ consists of sections that are smooth on each component $S_\alpha\subset\Sigma_u$ and continuous on the entire $\Sigma_u$.
\end{remark}

\begin{proposition}\label{prop:5.23}\label{Fix:M2:48}
  $\widetilde{E}_{J_0,u,\epsilon}\times \mathbb{C}^\ell$ as an obstruction datum defines an oriented Kuranishi structure
  \[\{(\widetilde{V}_u,\widetilde{E}_{J_0,u,\epsilon}\times \mathbb{C}^\ell,\Gamma_u,\widetilde{\psi}_u,\widetilde{s}_u\times \ev_{\mathbb{C},u})\}\]
  on $\overline{\mathcal{M}}_{0,k+\ell}(J_0,A,Y)$.
\end{proposition}

\begin{proof}
  Notice that $\bar{\partial}_J$ is transverse to zero on ghost spheres (see Lemma 6.7.6 in \cite{MS}), therefore
  \[\widetilde{E}_{J_0,u,\epsilon}\oplus \Img D_u\bar{\partial}_J=W^{r-1,p}(\Omega^{0,1}(\Sigma_T,u^{*}TM)).\]
  Take $\widetilde{E}_{J_0,u,\epsilon}\times\mathbb{C}^\ell$ as the obstruction data, and define $\widetilde{V}_u$ by $\widetilde{V}_{u,map}$, $\widetilde{V}_{u,resolve}$ and $\widetilde{V}_{u,deform}$ accordingly as in Section \ref{Sec:4}. The automorphism group $\Gamma_u=\{0\}$, the Kuranishi map
  \[\widetilde{s}_u(v)=\dbar_{J_0}(v)\times\ev_{\mathbb{C},u}(v),\ v\in \widetilde{V}_u,\]
  and the map
  \[\widetilde{\psi}_u:\widetilde{s}_u^{-1}(0)\rightarrow \widetilde{V}_u\cap \overline{\mathcal{M}}_{0,k+\ell}(J_0,A,Y)\]
  is the identity map. The rest of the proof is the same as in Theorem \ref{Thm:5.12}.
\end{proof}

\begin{proposition}\label{Prop:5.23}
  $\varphi_\ell$ extends to $\widetilde{V}_u$ for each $u\in \overline{\mathcal{M}}_{0,k+\ell}(J_0,A,Y)$, such that for $w\coloneqq \varphi_\ell(u)$ and $v\in \widetilde{V}_u$, we have $\varphi_\ell(v)\in V_w$.
\end{proposition}

\begin{proof}
  Denote by $\Sigma_u$ the underlying curve of $u$. Take a $v\in \widetilde{V}_u$ with underlying curve $\Sigma_v$. First, assume that $\Sigma_v=\Sigma_u$. We have
  \[\bar{\partial}_{J_0}v\equiv 0\ \mathrm{mod}\ \Par^{hol}_{uv}\widetilde{E}_{J_0,u,\epsilon}.\]
  If $S_\alpha\subset \Sigma_u$ is a ghost component contracted by the forgetful map $\underline{\varphi}_\ell$, then $S_\alpha$ contains some of the last $\ell$ marked points and
  \[\dbar_{J_0}v|_{S_\alpha}=0,\]
  this is because elements of $\Par^{hol}_{uv}\widetilde{E}_{J_0,u,\epsilon}$ vanish near $v^{-1}(Y)$. This indicates that $v|_{S_\alpha}$ is also a constant map. Hence $\underline{\varphi}_\ell$ contracts $S_\alpha\subset \Sigma_v$ as well, thus $\underline{\varphi}_\ell(\Sigma_v)=\underline{\varphi}_\ell(\Sigma_u)$.

  Let
  \[\bar{\partial}_{J_0} v=\eta,\quad \eta\in \Par_{uv}^{hol}\widetilde{E}_{J_0,u,\epsilon}\]
  and
  \[\bar{\partial}_{J_0} (\varphi_{\ell}(v))=\eta^{'}.\]
  For each non-ghost component $S_\alpha$ of $v$,
  \begin{equation}\label{eq:5.22}
  v|_{S_\alpha}=\varphi_\ell(v)|_{\underline{\varphi}_\ell(S_\alpha)},
  \end{equation}
  where we identify $\underline{\varphi}_\ell(S_\alpha)$ with $S_\alpha$ since $\underline{\varphi}_\ell$ does not change anything on non-ghost components except forgetting the last $\ell$ marked points. Notice that the parallel transport $\Par^{hol}$ is determined by the Levi-Civita connection $\nabla$ on $M$, and therefore on $S_\alpha$,
  \begin{equation}\label{eq:5.23}
  \Par^{hol}_{uv}|_{S_\alpha}=\Par^{hol}_{w\varphi_\ell(v)}|_{\varphi_\ell(S_\alpha)}.
  \end{equation}
  Equations (\ref{eq:5.22}) and (\ref{eq:5.23}) imply that
  \[(\Par^{hol}_{uv}|_{S_\alpha})^{-1}(\eta|_{S_\alpha})=(\Par^{hol}_{w\varphi_\ell(v)}|_{\varphi_\ell(S_\alpha)})^{-1}(\eta^{'}|_{\varphi_\ell(S_\alpha)})\]
  with respect to the identification between $S_\alpha$ and $\underline{\varphi}_\ell(S_\alpha)$. Since
  \[(\Par^{hol}_{uv})^{-1}(\eta)\in \widetilde{E}_{J_0,u,\epsilon},\]
  we have
  \[(\Par^{hol}_{w\varphi_\ell(v)})^{-1}(\eta^{'})\in E_{J_0,w,\epsilon}\]
  by the definition of $\widetilde{E}_{J_0,u,\epsilon}$. Thus
  \[\dbar_{J_0}\varphi_\ell(v)\equiv 0\mod \Par^{hol}_{w\varphi_\ell(v)}E_{J_0,w,\epsilon}.\]
  We conclude that $\varphi_\ell(v)\in V_w$.

  For a general $v\in \widetilde{V}_u$, the underlying curve is given by resolving nodal points and deforming marked points of $\Sigma_u$. Then the underlying curve of $\varphi_\ell(v)$ is obtained by resolving nodal points and deforming marked points of the underlying curve of $w$. \label{Fix:M2:49}Notice that the resolution only occurs where elements of $\widetilde{E}_{J_0,u,\epsilon}$ vanish; therefore, for the area of the underlying curve where elements of $\widetilde{E}_{J_0,u,\epsilon}$ do not vanish, the argument proceeds similarly as above. This completes the proof.
\end{proof}

\begin{definition}\label{Defi:5.21}
  Let $h_w$ be a multisection of the obstruction bundle $\mathcal{E}_{J,w,\epsilon}\rightarrow V_w$\label{sym:mEjwe} whose fiber at $\hat{w}\in V_w$ is the vector space $\Par_{w,\hat{w}}^{hol}E_{J_0,w,\epsilon}$. Assume that $w=\varphi_\ell(u)$. Let $\widetilde{\mathcal{E}}_{J,u,\epsilon}\rightarrow \widetilde{V}_u$ be the obstruction bundle whose fiber at $v\in \widetilde{V}_u$ is the vector space $\Par_{u,v}^{hol}\widetilde{E}_{J_0,u,\epsilon}$. 
  
  Firstly, we pull $\mathcal{E}_{J,w,\epsilon}\to V_w$ back to $\widetilde{V}_u$ via $\varphi_\ell$, which we denote the pullback vector bundel by $\varphi_\ell^{*}\mathcal{E}_{J,w,\epsilon}\to \widetilde{V}_u$, whose fiber at $v\in\widetilde{V}_u$ is the vector space $\Par_{w,\varphi_\ell(v)}^{hol}E_{J_0,w,\epsilon}$. For a multisection $h_w\in C^0(V_w,\mathcal{E}_{J,w,\epsilon})$, we have a pullback multisection $\varphi_\ell^{*}(h_w)\in C^0(\widetilde{V}_u,\varphi_\ell^{*}\mathcal{E}_{J,w,\epsilon})$. We define a map
  \[\varphi_{\ell,u}^{*}:C^0(V_w,\mathcal{E}_{J,w,\epsilon})\rightarrow C^0(\widetilde{V}_u,\widetilde{\mathcal{E}}_{J,u,\epsilon}),\]
  such that for any $v\in \widetilde{V}_u$,
  \[\varphi_{\ell,u}^{*}(h_w)(v)=\underline{\varphi}_\ell^{*}\left(\varphi_\ell^{*}(h_w)(v)\right)\in \Par_{u,v}^{hol}\widetilde{E}_{J_0,u,\epsilon}.\]
  We call $\varphi_{\ell,u}^{*}(h_w)$ the {\bf pullback multisection} on the obstruction bundle $\widetilde{\mathcal{E}}_{J_0,u,\epsilon}\rightarrow \widetilde{V}_u$.
\end{definition}

\begin{proposition}\label{Prop:5.25}
  The forgetful map $\varphi_\ell$ is a surjection from $\overline{\mathcal{M}}_{0,k+\ell}(J_0,A,Y)$ to $\overline{\mathcal{M}}_{0,k}(J_0,A)$.
\end{proposition}

\begin{proof}
  Let $w\in \overline{\mathcal{M}}_{0,k}(J_0,A)$ and $\Sigma_w$ be the underlying curve of $w$. Then $w$ intersects $Y$ at $\ell$ points (including tangency counts),
  \begin{enumerate}
    \item Suppose $y\in w(\Sigma_w)\cap Y$ such that $z\coloneqq w^{-1}(y)$ is a single point and $y$ has intersection number $1$ (which means transverse). Moreover $z$ does not coincide with any of the $k$ marked points or any nodal point, then we add one marked point at $z$. Otherwise, if $z$ is a nodal point $z_{\alpha\beta}$ connecting components $S_\alpha$ and $S_\beta$, then add a ghost sphere $S_z$ with one marked point and two nodal points connecting to $S_\alpha$ and $S_\beta$.
    \item If $y\in w(\Sigma_w)\cap Y$ such that $z\coloneqq w^{-1}(y)$ is a single point and $y$ has intersection number $\ell_y>1$ (which means tangency order $\ell_y-1$). Moreover $z$ does not coincide with any of the $k$ marked points or any nodal point, then we add a ghost bubble $S_z$ at $z$ with $\ell_y$ new marked points on $S_z$. Otherwise, if $z$ is a nodal point connecting sphere components $S_\alpha$ and $S_\beta$, then we construct a ghost bubble $S_z$ between $S_\alpha$ and $S_\beta$ with $\ell_y$ new marked points on $S_z$.
    \item If $y\in w(\Sigma_w)\cap Y$ such that $z\coloneqq w^{-1}(y)$ is a single point and $y$ has intersection number $\ell_y$, and $z$ coincides with one of the $k$ marked points $z_j$, then we construct a ghost bubble $S_z$ at $z_j$ with $\ell_y$ new marked points and a new $z_j$ on $S_z$.
    \item If $y\in w(\Sigma_w)\cap Y$ such that $w^{-1}(y)$ contains finite many points $z^1,\cdots, z^j$. 
    \begin{enumerate}
      \item For each $z^i$, $1\leq i\leq j-1$, if $z^i$ is not one of the $k$ marked points or a nodal point, then add a new marked point at $z^i$. Otherwise, if $z^i$ is a nodal point connecting components $S_\alpha$ and $S_\beta$, then add a ghost sphere $S_{z^i}$ with one new marked point and two nodal points connecting to $S_\alpha$ and $S_\beta$. If $z^i$ is one of the $k$ makred point, then add a ghost bubble $S_{z^i}$ at $z^i$ with two new marked points on it. 
      \item For $z^j$, if $z^j$ is not one of the $k$ marked points or a nodal point, and $\ell_y=j$, then add a new marked point at $z^j$. Otherwise, if $z^j$ is not one of the $k$ marked points or a nodal point but $\ell_y>j$, then add a ghost sphere $S_{z^j}$ at $z^j$ with $\ell_y-j+1$ new marked points on it. If $z^j$ is a nodal point connecting components $S_\alpha$ and $S_\beta$, then add a ghost sphere $S_{z^j}$ with $\ell_y-j+1$ new marked points and two nodal points connecting to $S_\alpha$ and $S_\beta$. If $z^j$ is one of the $k$ marked points, then add a ghost bubble $S_{z^j}$ at $z^j$ with $\ell_y-j+2$ new marked points on it.
    \end{enumerate}
    \item If $y\in w(\Sigma_w)\cap Y$ such that $w^{-1}(y)$ contains a ghost tree $T_{ghost}$, then put $\ell_y$ new marked points on the ghost tree away from the $k$ marked points and nodal points.
  \end{enumerate}
  Denote by $\Sigma_w^{'}$ the new underlying curve given by the procedures above. Extend $w$ to $\Sigma_w^{'}$ by defining it as a constant map on the newly added components. The map obtained by this extension (denote it by $u$) belongs to $\overline{\mathcal{M}}_{0,k+\ell}(J_0,A,Y)$ and $\varphi_\ell(u)=w$.
\end{proof}

\begin{proposition}\label{Prop:5.26}
  Let $u\in \overline{\mathcal{M}}_{0,k+\ell}(J_0,A,Y)$, $w\in \overline{\mathcal{M}}_{0,k}(J_0,A)$ and $\varphi_\ell(u)=w$. Then $\varphi_\ell:\widetilde{V}_u\rightarrow V_w$ is a surjection.
\end{proposition}
\begin{proof}
  Denote by $\Sigma_w$, $\Sigma_u$ the underlying curve of $w$ and $u$. Recall the following parametrizations in Section \ref{Sec:4},
  \begin{align*}
    &\widetilde{V}_u=\widetilde{V}_{u,map}\times \widetilde{V}_{u,resolve}\times \widetilde{V}_{u,deform},\\
    &V_w=V_{w,map}\times V_{w,resolve}\times V_{w,deform}.
  \end{align*}
  First, we consider $\widetilde{V}_{u,resolve}\times\widetilde{V}_{u,deform}$. Denote by $\widetilde{\bf z}_{k+\ell}=(\widetilde{z}_1,\ldots,\widetilde{z}_{k+\ell})$ the $k+\ell$ marked points of $\Sigma_u$, and ${\bf z}_k=(z_1,\ldots,z_k)$ the $k$ marked points of $\Sigma_w$. Using the stabilization data defined in \cite[Def.~17.7~(1)(8)]{FOOO-T}, we add marked points $\widetilde{\bf z}_{\mathfrak{l}}=(\widetilde{z}_{k+\ell+1},\ldots,\widetilde{z}_{k+\ell+\mathfrak{l}})$ to $\Sigma_u$, and ${\bf z}_{\mathfrak{l}}=(z_{k+1},\ldots,z_{k+\mathfrak{l}})$ to $\Sigma_w$ such that,
  \begin{enumerate}
    \item $z_{k+i}=\underline{\varphi}_\ell(\widetilde{z}_{k+\ell+i})$ on $\Sigma_w$ and $\Sigma_u$ for $i=1,\ldots,\mathfrak{l}$.
    \item $(\Sigma_u,\widetilde{\bf z}_{k+\ell}\cup\widetilde{\bf z}_{\mathfrak{l}})$ is stable.
    \item $(\Sigma_w,{\bf z}_k\cup{\bf z}_{\mathfrak{l}})$ is stable.
  \end{enumerate}
  With these settings, the forgetful map $\underline{\varphi}_\ell$ is the forgetful map between Deligne-Mumford spaces $\overline{\mathcal{M}}_{0,k+\ell+\mathfrak{l}}$ and $\overline{\mathcal{M}}_{0,k+\mathfrak{l}}$ by forgetting the middle $\ell$ marked points of curves in $\overline{\mathcal{M}}_{0,k+\ell+\mathfrak{l}}$, which is a surjection. $\widetilde{V}_{u,resolve}\times\widetilde{V}_{u,deform}$ is a submanifold of $\overline{\mathcal{M}}_{0,k+\ell+\mathfrak{l}}$, which contains $\Sigma_u$ and fixes the last $\mathfrak{l}$ marked points. $V_{w,resolve}\times V_{w,deform}$ is a submanifold of $\overline{\mathcal{M}}_{0,k+\mathfrak{l}}$, which contains $\Sigma_w$ and fixes the last $\mathfrak{l}$ marked points. Since $\underline{\varphi}_\ell(\Sigma_u)=\Sigma_w$, $z_{k+i}=\underline{\varphi}_\ell(\widetilde{z}_{k+\ell+i})$ for $i=1,\ldots,\mathfrak{l}$, we conclude that $\underline{\varphi}_\ell:\widetilde{V}_{u,resolve}\times\widetilde{V}_{u,deform}\rightarrow V_{w,resolve}\times V_{w,deform}$ is surjective.

  As for $\widetilde{V}_{u,map}$, notice that (see Definition \ref{Def:4.22.2})
  \[\widetilde{V}_{u,map}\cong \ker(\pi_{J_0,u}\circ D_u\bar{\partial}_{J_0}),\  V_{w,map}\cong\ker(\pi_{J_0,w}\circ D_w\bar{\partial}_{J_0}),\]
  where
  \[\pi_{J_0,u}:W^{r-1,p}(\Omega^{0,1}(\Sigma_u,u^{*}TM))\rightarrow \frac{W^{r-1,p}(\Omega^{0,1}(\Sigma_u,u^{*}TM))}{E_{J_0,u}},\]
  $\pi_{J_0,w}$ is defined in a similar way for $w$. For each $v\in \ker(\pi_{J_0,u}\circ D_u\bar{\partial}_{J_0})$, the under lying curve of $v$ is also $\Sigma_u$. We observe that:
  \begin{enumerate}
    \item If a sphere component $S_\alpha\subset \Sigma_u$ is contracted by $\underline{\varphi}_\ell$, then $S_\alpha$ contains some of the last $\ell$ marked points, thereby the elements of $\widetilde{E}_{J_0,u,\epsilon}$ vanish on it (see Remark \ref{RMK:5.18}). Moreover, $v$ is also a constant map on $S_\alpha$.
    \item If $S_\alpha$ is not contracted by $\underline{\varphi}_\ell$, then $\widetilde{E}_{J_0,u,\epsilon}\cong E_{J_0,w,\epsilon}$ on $S_\alpha$ (see Remark \ref{RMK:5.18}), then
    \[(\pi_{J_0,w}\circ D_w\bar{\partial}_{J_0})(\varphi_\ell(v))|_{S_\alpha}=(\pi_{J_0,u}\circ D_u\bar{\partial}_{J_0})(v)|_{S_\alpha}=0.\]
  \end{enumerate}
  Therefore $\varphi_\ell:\ker(\pi_{J_0,u}\circ D_u\bar{\partial}_{J_0})\to \ker(\pi_{J_0,w}\circ D_w\bar{\partial}_{J_0})$ is a surjection. Combining our discussion for $\underline{\varphi}_\ell:\widetilde{V}_{u,resolve}\times\widetilde{V}_{u,deform}\rightarrow V_{w,resolve}\times V_{w,deform}$ above, we conclude that $\varphi_\ell$ maps $\widetilde{V}_u$ onto $V_w$.

\end{proof}

\subsubsection{Submanifolds of the Kuranishi neighborhoods}
\begin{remark}\label{RMK:5.23}
  Suppose $w\in \overline{\mathcal{M}}_{0,k}(J_0,A)$ has the underlying curve $\Sigma_w$. Recall that the neighborhood $V_w$ is parameterized by
  \[(\nu,\mathfrak{a},\vartheta)\in V_{w,map}\times V_{w,resolve}\times V_{w,deform},\]
  where (see Section \ref{Sec:4})
  \[\mathfrak{a}=(\mathfrak{a}_z)\in\bigoplus_{\substack{z=z_{\alpha\beta}=z_{\beta\alpha}\\ \text{ is a nodal point of }\Sigma_w}}T_{z_{\alpha\beta}}S_\alpha\otimes T_{z_{\beta\alpha}}S_\beta.\]
  We list the nodal points of $\Sigma_w$ by $(z_{nod,1},\ldots,z_{nod,\mathfrak{n}})$. For a $k$-labeled tree $T$, denote by
  \[U_{w,T}\coloneqq \{\hat{w}\in V_w|\hat{w}\text{ is modeled on tree }T\}.\]
  If $U_{w,T}\neq \emptyset$, then there is a subset $I_T\subset\{1,\ldots,\mathfrak{n}\}$, such that $U_{w,T}$ is parametrized by $(\nu,(\mathfrak{a}_{z_{nod,i}}),\vartheta)$, where $\mathfrak{a}_{z_{nod,i}}= 0$ if and only if $i\in I_T$.

  Therefore $U_{w,T}$ {\bf is a smooth submanifold} of $V_w$. If $T$ has more than one node, then $U_{w,T}$ is of real codimension at least $2$, since there will be at least one $\mathfrak{a}_{z_{nod,i}}= 0$ in this case.
\end{remark}

Next, fix a $w\in\overline{\mathcal{M}}_{0,k}(J_0,A)$ and the Kuranishi neighborhood $V_w$, we will establish a result related to the space of maps that have tangency with $Y$. This is a simpler variation of Section 6 in \cite{CM}.

Denote by $V_{w;o}$ the open subset of $V_w$ such that
\[V_{w;o}\coloneqq \{v\in V_w|v\not\in\bigcup_{\substack{T\text{ with more}\\\text{than one}\\\text{ node }}}U_{w,T}\}.\]
In other words, $v\in V_{w;o}$ are maps with only one node. Let $\Sigma_v\cong S^2$ be the underlying curve of $v\in V_{w;o}$, and $\mathfrak{z}\in \Sigma_v$ be such that $v(\mathfrak{z})\in Y$.

\begin{lemma}[{\cite[Lem.~6.6]{CM}}]\label{lemma:5.25}
  For $r-2/p>1$, $v\in V_{w;o}$ with the obstruction bundle fiber $E_{J_0,v}$. Denote by
  \[\pi_{J_0,v}\coloneqq W^{r-1,p}(\Omega^{0,1}(S^2,v^{*}TM))\rightarrow W^{r-1,p}(\Omega^{0,1}(S^2,v^{*}TM))/E_{J_0,v}\]
  the projection. Define the spaces
  \begin{align*}
    &B_0^{r,p}\coloneqq\{\xi|\xi\in W^{r,p}(S^2,v^{*}TM),\ d\xi(\mathfrak{z})=0\},\\
    &E_0^{r-1,p}\coloneqq \{\eta|\eta\in W^{r-1,p}(\Omega^{0,1}(S^2,v^{*}TM)),\ \eta(\mathfrak{z})=0\},\\
    &E_{0,quo}^{r-1,p}\coloneqq E_0^{r-1,p}/E_{J_0,v}.
  \end{align*}
  Then the linear operator $F_0\coloneqq\pi_{J_0,v}\circ D_v\dbar_{J_0}:B_0^{r,p}\rightarrow E_{quo}^{r-1,p}$ is surjective.
\end{lemma}

\begin{proof}\label{Fix:M2:51}
  Firstly, assume that $1<r-2/p<2$. Denote by $W^{r-1,p*}(\Omega^{0,1}(S^2,v^{*}TM))$ the dual space of $W^{r-1,p}(\Omega^{0,1}(S^2,v^{*}TM))$. The following space
  \[W^{r-1,p*}_{v,quo}\coloneqq \{\zeta|\zeta\in W^{r-1,p*}(\Omega^{0,1}(S^2,v^{*}TM)),\ \zeta\text{ vanishes on }E_{J_0,v}\}\]
  is the dual of $W^{r-1,p}(\Omega^{0,1}(S^2,v^{*}TM))/E_{J_0,v}$.

  Assume that $\zeta\in W^{r-1,p*}_{v,quo}$ and $\zeta$ vanishes on $F_0(B_0^{r,p})$, that is
  \[\zeta(D_v\dbar_{J_0}\xi)=0\]
  for all $\xi\in B_0^{r,p}$. In particular, this implies that $\zeta(D_v\dbar_{J_0}\xi)=0$ for all $\xi$ with compact support in $S^{*}\coloneqq S^2\setminus\{\mathfrak{z}\}$. According to elliptic regularity for distributions (see Theorem 8.12 in \cite{Rudin}), the restriction of $\zeta$ on $S^{*}$ can be represented by a smooth section $\zeta_{S^{*}}:S^{*}\rightarrow \Omega^{0,1}(S^2,v^{*}TM)$, such that
  \[\zeta(\hat{\eta})=\langle \hat{\eta},\zeta_{S^{*}}\rangle_{L^2}\]
  for all $\hat{\eta}\in C^\infty_0(S^{*},\Omega^{0,1}(S^2,v^{*}TM))$.

  Recall that
  \[\pi_{J_0,v}\circ D_v\dbar_{J_0}:W^{r,p}(S^2,v^{*}TM)\rightarrow W^{r-1,p}(\Omega^{0,1}(S^2,v^{*}TM))/E_{J_0,v}\]
  is surjective, thus there is a $\widetilde{\xi}\in W^{r,p}(S^2,v^{*}TM)$ and a $f\in E_{J_0,v}$ such that
  \[D_v\dbar_{J_0}(\widetilde{\xi})=\hat{\eta}+f.\]
  Since $B_0^{r,p}$ is dense in $W^{1,p}(S^2,v^{*}TM)$, there is a sequence $\widetilde{\xi}_i\in B_0^{r,p}$ such that
  \[\lim_{i\rightarrow \infty}\|\widetilde{\xi}_i-\widetilde{\xi}\|_{W^{1,p}}=0,\]
  and we have
  \begin{align*}
    &\zeta(\hat{\eta})=\zeta(D_v\dbar_{J_0}(\widetilde{\xi}))=\langle D_v\dbar_{J_0}(\widetilde{\xi}),\zeta_{S^{*}}\rangle_{L^2}\\
    =&\lim_{i\rightarrow \infty}\langle D_v\dbar_{J_0}(\widetilde{\xi}_i),\zeta_{S^{*}}\rangle_{L^2}=\lim_{i\rightarrow \infty}\zeta(D_v\dbar_{J_0}(\widetilde{\xi}_i))=0.
  \end{align*}
  This is true for any $\hat{\eta}\in C^\infty_0(S^{*},\Omega^{0,1}(S^2,v^{*}TM))$, therefore $\zeta$ is a distribution supported at $\mathfrak{z}$. According to Theorem 6.25 in \cite{Rudin}, $\zeta$ has the following form
  \[\zeta=\sum_{|\alpha|\leq N}c_\alpha D^\alpha\delta_{\mathfrak{z}},\]
  \label{Fix:M2:52}where $c_\alpha:\mathbb{C}^n\rightarrow \mathbb{R}$ are $\mathbb{R}$-linear functions, $\delta{\mathfrak{z}}$ is the distribution defined by
  \[\delta{\mathfrak{z}}(\eta)=\eta(0),\ \eta\in C^\infty(\Omega^{0,1}(S^2,v^{*}TM)),\]
  $D^\alpha$ is the derivative of multi-indices $\alpha$, such that
  \[D^\alpha\delta_{\mathfrak{z}}(\eta)=(-1)^{|\alpha|}D^\alpha\eta(\mathfrak{z}),\ \eta\in C^\infty(\Omega^{0,1}(S^2,v^{*}TM)).\]
  $\zeta\in W^{r-1,p*}(\Omega^{0,1}(S^2,v^{*}TM))$ implies that for $c_\alpha\neq 0$, there is
  \[|D^\alpha\eta(\mathfrak{z})|\leq C\|\eta\|_{W^{r-1,p}},\ \forall \eta\in C^\infty(\Omega^{0,1}(S^2,v^{*}TM)).\]
  According to the Sobolev embedding, $|\alpha|\leq r-1-2/p$. However, our assumption further implies that $|\alpha|<1$, which means
  \[\zeta=c_0\delta_{\mathfrak{z}} \]
  and
  \[\zeta(\eta)=c_0\eta(0)=0,\ \forall \eta\in E_0^{r-1,p}.\]

  From the argument above, we conclude that if $\zeta\in W^{r-1,p*}_{v,quo}$ vanishes on $F_0(B_0^{r,p})$, then $\zeta$ vanishes on $E_{0,quo}^{r-1,p}$ as well. Therefore $F_0$ is a surjection for $1<r-2/p<2$. The general case of $r,p$ with $r-2/p>1$ follows from elliptic regularity as in Lemma 6.6 of \cite{CM}.
\end{proof}

Denote by $\pi_{TM/TY}:TM\rightarrow TM/TY\cong \mathbb{C}$ the projection map, and
\[j_Y^1\ev_{\mathfrak{z}}:V_{w;o}\rightarrow TM/TY,\ v\mapsto \pi_{TM/TY}\circ dv(\mathfrak{z})\]
the normal $1$-jet evaluation map to the normal bundle $TM/TY$. The linearization of $j_Y^1\ev_{\mathfrak{z}}$ at $v\in V_{w;o}$ is the linear operator
\[L_{\mathfrak{z}}:T_vV_w\rightarrow \mathbb{C}^n/\mathbb{C}^{n-1}\cong \mathbb{C},\ \xi\mapsto \pi_{TM/TY}\circ d\xi(\mathfrak{z}).\]

\begin{lemma}[{\cite[Lem.~6.5]{CM}}]\label{lemma:5.26}\label{sym:JY1evfrakz}
  For $r-2/p>1$ and $v\in V_{w;o}$, the map $L_{\mathfrak{z}}$ is surjective.
\end{lemma}

\begin{proof}
  Take a neighborhood $O\subset \mathbb{C}$ around $0$, then for any $\nu\in \mathbb{C}^n$ there is a holomorphic section $\widetilde{\xi}$ of the vector bundle $\mathbb{C}^n\times O\rightarrow O$
  \[\widetilde{\xi}(z)\coloneqq z\nu,\ z\in\mathbb{C}.\]
  We have
  \[\widetilde{\xi}(0)=0,\ d\widetilde{\xi}(0)=\nu.\]
  Regard $O$ as a neighborhood of $\mathfrak{z}\in\Sigma_v\cong S^2$, the bundle $\mathbb{C}^n\times O\rightarrow O$ as a neighborhood of $v^{*}TM$ with respect to the trivialization $\varrho$. We extend $\widetilde{\xi}$ to a section of $W^{r,p}(\Sigma_v,v^{*}TM)$.

  Denote by
  \[\widetilde{\eta}\coloneqq D_v\dbar_{J_0}\widetilde{\xi}.\]
  $\widetilde{\eta}$ represents an element $[\widetilde{\eta}]\in E_0^{r-1,p}$ defined in Lemma \ref{lemma:5.25}.  According to Lemma \ref{lemma:5.25}, $D_v\dbar_{J_0}:B_0^{r,p}\rightarrow E_0^{r-1,p}$ is a surjection, thus there is a $\widetilde{\xi}^{'}\in B_0^{r,p}$ such that
  \[D_v\dbar_{J_0}(\widetilde{\xi}^{'})=[\widetilde{\eta}].\]
  Denote by $\xi\coloneqq \widetilde{\xi}-\widetilde{\xi}^{'}$, we have
  \[D_v\dbar_{J_0}(\xi)=0,\]
  therefore $\xi\in T_vV_w$. Moreover,
  \[d\xi=d\widetilde{\xi}-d\widetilde{\xi}^{'}=\nu.\]
  This completes the proof.
\end{proof}

\begin{proposition}\label{Prop:5.27}
  For $r-2/p>1$, the following space
  \[U_{Y,w,tan}\coloneqq\{v|v\in V_{w,o},\ v\in j_Y^1\ev_{\mathfrak{z}}^{-1}(0)\}\]
  is a smooth submanifold in $V_w$ of real codimension at least $2$.
\end{proposition}

\begin{proof}
  According to Lemma \ref{lemma:5.26}, the linearization of $j_Y^1\ev_{\mathfrak{z}}$ is a surjection to $\mathbb{C}$, therefore $j_Y^1\ev_{\mathfrak{z}}^{-1}(0)$ is a smooth submanifold of codimension at least $2$.
\end{proof}

\begin{proposition}\label{Prop:5.24}
  For $w\in\overline{\mathcal{M}}_{0,k}(J_0,A)$ and $V_w$, let $\ev_j:V_w\rightarrow M$ be the evaluation map of the $j$-th marked point, where $j=1,\ldots,k$. Then on each $V_w$,
  \[(\ev_j^{-1}(Y)\cap V_w)\setminus \left(U_{Y,w,tan}\cup\bigcup_{\substack{T\text{ more than}\\\text{one node}}}U_{w,T}\right)\]
  is a submanifold of $V_w$ with real codimension at least $2$. We denote the collection of these submanifolds as $\{U_{Y,w,j}\}_j$.
\end{proposition}

\begin{proof}
  Any map $\hat{w}$ in
  \[V_w\setminus \left(U_{Y,w,tan}\cup\bigcup_{\substack{T\text{ more than}\\\text{one node}}}U_{w,T}\right)\]
  has only one node and is transverse to $Y$. Therefore $\ev_j$ is transverse to $Y$ at $\hat{w}$. This completes the proof.
\end{proof}

Recall that, the Kuranishi structure
\[\{(\widetilde{V}_u,\widetilde{E}_{J_0,u,\epsilon}\times \mathbb{C}^\ell,\Gamma_u,\widetilde{\psi}_u,\widetilde{s}_u\times \ev_{\mathbb{C},u})\}_{u\in\overline{\mathcal{M}}_{0,k+\ell}(J_0,A,Y)}\]
defined by Proposition \ref{prop:5.23} and Theorem \ref{Thm:5.12} depends on $\epsilon>0$, so that
\[\prod_{i=k+1}^{k+\ell}\ev_i(\widetilde{V}_u)\subset N_{Y,\epsilon}^\ell,\]
where $N_{Y,\epsilon}$ is the $\epsilon$-tubular neighborhood of $Y$ in $M$.

\begin{lemma}\label{lemma:5.25}
  For $w\in \overline{\mathcal{M}}_{0,k}(J_0,A)$ with the Kuranishi structure
  \[\{(V_w,E_{J_0,w,\epsilon},\Gamma_w,\psi_w,s_w)\}_{w\in\overline{\mathcal{M}}_{0,k}(J_0,A)},\]
  assume that $\varphi_\ell(u)=w$ where $u\in\overline{\mathcal{M}}_{0,k+\ell}(J_0,A,Y)$. Denote by
  \begin{align*}
    V_{w,\epsilon}^{\bot}\coloneqq V_w\setminus\left(\bigcup_{\substack{T\text{ with more}\\\text{than one}\\\text{ node }}}U_{w,T}\cup U_{Y,w,tan}\cup \bigcup_{i=1}^k\ev_i^{-1}(N_{Y,\epsilon}) \right),
  \end{align*}
  and
  \[\widetilde{V}_{u,\epsilon}^{\bot}\coloneqq \varphi_\ell^{-1}(V_{w,\epsilon}^{\bot})\subset \widetilde{V}_u.\]
  Then $\varphi_\ell:\widetilde{V}_{u,\epsilon}^{\bot}\to V_{w,\epsilon}^{\bot}$ is a submersion.
\end{lemma}

\begin{proof}
  Every map $\hat{w}\in V_{w,\epsilon}^{\bot}$ has the following properties:
  \begin{enumerate}
    \item The underlying curve $\Sigma_{\hat{w}}$ of $\hat{w}$ has only one node;
    \item All $k$ marked points of $\hat{w}$ are away from $Y$ with a positive distance $\epsilon$;
    \item $\hat{w}$ intersects $Y$ transversely.
  \end{enumerate}
  Hence a preimage in $\varphi_\ell^{-1}(\hat{w})$ is obtained by just adding distinct $\ell$ marked points to $\Sigma_{\hat{w}}$, and the new marked points are away from the original $k$ marked points as well. As a consequence, $\varphi_\ell^{-1}(\hat{w})$ consists of maps with only one node. For $v\in\widetilde{V}_u$ and $v\in\varphi_\ell^{-1}(\hat{w})$, since sufficiently small neighborhoods of one node maps $v$ and $\hat{w}$ in $\widetilde{V}_u$ and $V_w$ are diffeomorphic via $\varphi_\ell$, $\hat{w}$ is a regular value of the forgetful map $\varphi_\ell$.
\end{proof}

\subsubsection{Compare virtual fundamental classes on $\overline{\mathcal{M}}_{0,k+\ell}(J_0,A,Y)$ and $\overline{\mathcal{M}}_{0,k}(J_0,A)$}

\begin{remark}\label{RMK:5.29}
  Let $X$ be a finite dimensional compact (possibly with boundary) smooth manifold, $Y\subset X$ a smooth submanifold, and $E\times X\rightarrow X$ a finite dimensional vector bundle with the vector space $E$ as its fiber at each point of $X$. Denote by $\iota_Y:C^\infty(X,E)\rightarrow C^\infty(Y,E)$ the restriction map. Let $B\subset C^\infty(Y,E)$ be an open dense subset, then we claim that {\bf $\iota_Y^{-1}(B)$ is an open dense subset of $C^\infty(X,E)$}.

  Since $\iota_Y$ is continuous, it is clear that $\iota_Y^{-1}(B)$ is open if $B$ is open.

  Take a metric $\mathfrak{g}(\cdot,\cdot)$ on $X$, two tubular neighborhoods $N_XY^{'}\subsetneqq N_XY$ of $Y$, and a finite collection of open sets $\{O_i\subset X\}$ covering $N_XY$. Moreover, on each $O_i$, there is a trivialization $({\bf y}^i, x^i)$ where ${\bf y}^i$ parametrizes $O_i\cap Y$, and $x^i$ is the normal direction with respect to the metric $\mathfrak{g}$. We also take a finite set of smooth functions $\{\chi_i\in C^\infty_0(O_i,[0,1])\}$, such that $\chi_i|_{N_XY^{'}}$ are partition of unity of $N_XY^{'}$.

  For each $f\in C^\infty(X,E)$, denote its restriction to $Y$ by $f|_Y\in C^\infty(Y,E)$. For any $\hat{h}\in C^\infty(Y,E)$ in a neighborhood of $f|_Y$, we extend $\hat{h}$ to an $h\in C^\infty(X,E)$ in the following way,
  \begin{enumerate}
    \item $h_i({\bf y}^i,x^i)\coloneqq \hat{h}({\bf y}^i)-f|_Y({\bf y}^i)+f({\bf y}^i,x^i)$ on $O_i$;
    \item Take $h\coloneqq \sum_i(\chi_ih_i+(1-\chi_i)f)$.
  \end{enumerate}
  Denote this extension by $\mu_{Y;f}:\hat{h}\mapsto h$. We have
  \begin{enumerate}
    \item $h|_Y=\hat{h}$;
    \item $\|\partial_\alpha (h-f)({\bf y^i},x^i)\|\leq \sum_{|\beta|\leq |\alpha|}c_{\beta,\chi}\|\partial_\beta(\hat{h}-\hat{f})({\bf y}^i)\|$ for each $({\bf y}^i,x^i)\in O_i$ and each $O_i$.
  \end{enumerate}
  Here $\alpha$ are multi-indices with respect to the parameters $({\bf y}^i,x^i)$, $\beta$ are multi-indices with respect to ${\bf y}^i$, $c_{\beta,\chi}$ are constants depending on $\beta$ and $\{\chi_i\}$. Notice that
  \[h-f\equiv 0,\text{ on }X\setminus \cup_i O_i.\]
  We conclude that the extension map $\mu_{Y;f}:C^\infty(Y,E)\rightarrow C^\infty(X,E)$ is continuous.

  Assume that $B\subset C^\infty(Y,E)$ is dense. Then for any open neighborhood $B_f\subset C^\infty(X,E)$ around $f$, $\mu_{Y;f}^{-1}(B_f)$ is open in $C^\infty(Y,E)$ and contains $f|_Y$. Therefore there is a $\hat{h}\in \mu_{Y;f}^{-1}(B_f)\cap B$ such that $\mu_{Y;f}(\hat{h})\in B_f\cap \iota_Y^{-1}(B)$. Hence $\iota_Y^{-1}(B)$ is dense in $C^\infty(X,E)$. This completes the proof.
\end{remark}

Recall the definitions of $U_{w,T}$, $U_{Y,w,j}$ and $U_{Y,w,tan}$ in Remark \ref{RMK:5.23} and Proposition \ref{Prop:5.24}, \ref{Prop:5.27}.
\begin{corollary}\label{Coro:5.30}
  On the Kuranishi structure
  \[\{(V_w,E_{J_0,w,\epsilon},\Gamma_w,\psi_w,s_w)\}_{w\in\overline{\mathcal{M}}_{0,k}(J_0,A)},\]
  there exists a multisection $h=\{h_w\}$, such that $h$ is transverse to $0$, and $h_w$ restricted to $U_{w,T}$, $U_{Y,w,j}$ and $U_{Y,w,tan}$ is transverse to $0$ as well.
\end{corollary}

\begin{proof}
  Recall from Theorem 6.4 in \cite{FK} the construction of $h=\{h_w\}$ that makes $h$ transverse to $0$. We take a good coordinate system $(P,((V_{w_i},\psi_{w_i},s_{w_i}):w_i\in P),\phi_{w_iw_j},\hat{\phi}_{w_iw_j})$ as \cite[Def.~6.1]{FK}, where $j\leq i$ if and only if $\mathrm{rank}E_{J_0,w_j}\leq \mathrm{rank}E_{J_0,w_i}$. On $V_{w_i}$, each branch of $h_w$ is taken from an open dense subset $\mathcal{B}\subset C^\infty(V_{w_i},\mathcal{E}_{J_0,w_i})$.

  On the other hand, there exist open dense subsets $\mathcal{B}_{w_i,T}\subset C^\infty(U_{w_i,T}, \mathcal{E}_{J_0,w_i})$, $\mathcal{B}_{Y,w_i,j}\subset C^\infty(U_{Y,w_i,j}, \mathcal{E}_{J_0,w_i})$ and $\mathcal{B}_{Y,w_i,tan}\subset C^\infty(U_{Y,w_i,tan}, \mathcal{E}_{J_0,w_i})$, such that sections in $\mathcal{B}_{w_i,T}$, $\mathcal{B}_{Y,w_i,j}$ and $\mathcal{B}_{Y,w_i,tan}$ are transverse to $0$.

  For a submanifold $U\subset V_{w_i}$, denote by $\iota_U:C^\infty(V_{w_i},\mathcal{E}_{J_0,w_i})\rightarrow C^\infty(U,\mathcal{E}_{J_0,w_i})$ the restriction map. Then according to Remark \ref{RMK:5.29},
  \[\bigcap_{i,T} \iota_{U_{w_i,T}}^{-1}(\mathcal{B}_{w_i,T}),\ \bigcap_{i,j} \iota_{U_{Y,w_i,j}}^{-1}(\mathcal{B}_{Y,w_i,j}),\ \bigcap_i \iota_{U_{Y,w_i,tan}}^{-1}(\mathcal{B}_{Y,w_i,tan})\]
  are all open dense subsets of $C^\infty(V_{w_i},\mathcal{E}_{J_0,w_i})$. By taking $h_{w_i}$ from the intersection of the above open dense subsets and $\mathcal{B}$, we complete the proof.
\end{proof}

\begin{theorem}\label{Thm:5.22}
  Let $(M,\omega)$ be a $2n$ dimensional smooth symplectic manifold, and $J$ an $\omega$-compatible almost complex structure, $[\omega]\in H^2(M;\mathbb{Z})$, and $A\in H_2(M;\mathbb{Z})$. $(J,Y)$ is a Donaldson pair of sufficiently large degree $D$ and $\ell=D\omega(A)\geq 3$. Take a domain independent
  \[J_0\in B^{*}\subset \mathcal{J}_{\ell+1}^{*}(M,Y;J,\theta_1),\]
  where $B^{*}$ and $\mathcal{J}_{\ell+1}^{*}(M,Y;J,\theta_1)$ are defined in Definition \ref{Def:3.20}. Take $\epsilon>0$, then the virtual fundamental class provided by ${\bf ev}_k=\prod_{i=1}^{k}\ev_i$ with the Kuranishi structure
  \[\{(\widetilde{V}_u,\widetilde{E}_{J_0,u,\epsilon}\times \mathbb{C}^\ell,\Gamma_u,\widetilde{\psi}_{u},\widetilde{s}_u\times \ev_{\mathbb{C},u})\}_{u\in \overline{\mathcal{M}}_{0,k+\ell}(J_0,A,Y)}\]
  is $\ell!$ times the virtual fundamental class given by ${\bf ev}_k$ with
  \[\{(V_w,E_{J_0,w,\epsilon},\Gamma_w,\psi_w,s_w)\}_{w\in\overline{\mathcal{M}}_{0,k}(J_0,A)}.\]
\end{theorem}

\begin{proof}
  \textbf{Step 1. \textit{Settings.}} Fix $\epsilon>0$. Consider the Kuranishi structure
  \[\{(V_w,E_{J_0,w,\epsilon},\Gamma_w,\psi_w,s_w)\}_{w\in\overline{\mathcal{M}}_{0,k}(J_0,A)}.\]
  Take a good coordinate system $(P,((V_{w_i},\psi_{w_i},s_{w_i}):w_i\in P),\phi_{w_iw_j},\hat{\phi}_{w_iw_j})$, where $j\leq i$ if and only if $\mathrm{rank}E_{J_0,w_j}\leq \mathrm{rank}E_{J_0,w_i}$. Let $h=\{h_{w_i}\}$ be a generic multisection of the obstruction bundle $\mathcal{E}_{J,w_i,\epsilon}\rightarrow V_{w_i}$ whose fiber at $\hat{w}\in V_{w_i}$ is the vector space $\Par_{w_i,\hat{w}}^{hol}E_{J_0,w_i,\epsilon}$. Moreover, for each $k$-labeled tree $T$, $w_i$ and $1\leq j\leq k$, we require that $h_{w_i}$ restricted to $U_{w_i,T}$, $U_{Y,w_i,j}$ and $U_{Y,w_i,tan}$ is transverse to $0$ as well (see definitions of $U_{w_i,T}$, $U_{Y,w_i,j}$ and $U_{Y,w_i,tan}$ in Remark \ref{RMK:5.23}, Proposition \ref{Prop:5.24}, \ref{Prop:5.27}). This is possible according to Corollary \ref{Coro:5.30}.

  \textbf{Step 2. \textit{Take neighborhoods covering the images of "bad" parts.}} For any $V_{w_i}$, assume that each branch of $h_{w_i}^{-1}(0)$ is a smooth submanifold of dimension $d=\dim V_{w_i}-\mathrm{rank}E_{J_0,w_i,\epsilon}$. Since $U_{w_i,T}$, $U_{Y,w_i,j}$ and $U_{Y,w_i,tan}$ have codimension at least $2$, the zero locus of $h_{w_i}$ restricted to $U_{w_i,T}$, $U_{Y,w_i,j}$ and $U_{Y,w_i,tan}$ consists of finitely many smooth submanifolds of dimension at most $d-2$. By Proposition \ref{Prop:3.3.1}, there exists an open neighborhood $U$ in $M^k$ such that
  \[\bigcup_i {\bf ev}_k\left(h_{w_i}^{-1}(0)\cap \left( U_{w_i,T}\cup \bigcup_{j=1}^k U_{Y,w_i,j}\cup U_{Y,w_i,tan}\right)\right)\subset U\]
  and
  \[H_l(U)=0,\quad\text{if }l>d-2.\]
  We fix such an open neighborhood $U$. For each $i$, the set ${\bf ev}_k^{-1}(U)$ is an open neighborhood of $(h_{w_i}^{-1}(0)\cap\bigcup_{j=1}^k U_{Y,w_i,j})$ in $V_{w_i}$. Since the entire $h^{-1}(0)$ is a compact set, then image ${\bf ev}_k(h^{-1}(0)\setminus {\bf ev}_k^{-1}(U))$ is a compact set disjoint from the $k$-product of the Donaldson hypersurface $Y$. Thus there exists a constant $\epsilon_1>0$ such that for any $i$, any $\hat{w}\in h_{w_i}^{-1}(0)\setminus {\bf ev}_k^{-1}(U)$ and any marked point $z_j$ of $\hat{w}$, we have that $\hat{w}(z_j)$ lies at distance at least $2\epsilon_1>0$ from the Donaldson hypersurface $Y$.

  For the subsequent steps, we need a smaller open neighborhood $U_1\subset U$ with similar properties to $U$. Note that if $d>\dim M^k$, then $H_d(M^k)=0$ and the theorem is trivial. If $d\leq \dim M^k$, we first notice that, by the definitions of $U_{w_i,T}$, $U_{Y,w_i,j}$ and $U_{Y,w_i,tan}$, the loci $h^{-1}(0)$ intersecting these sets are precisely the preimages of $Y$ and $0$ within $h^{-1}(0)$ under certain continuous maps (see Remark \ref{RMK:5.23}, Proposition \ref{Prop:5.24} and Proposition \ref{Prop:5.27}). Consequently, the intersections $(h^{-1}(0)\cap \bigcup_i U_{w_i,T})$, $(h^{-1}(0)\cap \bigcup_{i,j} U_{Y,w_i,j})$ and $(h^{-1}(0)\cap \bigcup_i U_{Y,w_i,tan})$ are compact subsets of $h^{-1}(0)$, and hence their images under ${\bf ev}_k$ are compact and proper subsets of the open set $U$. Therefore, we can choose a smaller open neighborhood $U_1$ such that its closure $\overline{U}_1\subset U$, and
  \[\bigcup_i {\bf ev}_k\left(h_{w_i}^{-1}(0)\cap \left( U_{w_i,T}\cup \bigcup_{j=1}^k U_{Y,w_i,j}\cup U_{Y,w_i,tan}\right)\right)\subset U_1\subset U\]
  and for any $i$, any $\hat{w}\in h_{w_i}^{-1}(0)\setminus {\bf ev}_k^{-1}(U_1)$ and any marked point $z_j$ of $\hat{w}$, we have that $\hat{w}(z_j)$ lies at distance at least $\epsilon_1>0$ from $Y$.
  
  \textbf{Step 3. \textit{Perturb the pullback multisection.}} For $w=\varphi_\ell(u)$, denote by
  \[\widetilde{h}_u\coloneqq \varphi_{\ell,u}^{*}(h_w)\]
  the pullback multisection on the obstruction bundle $\widetilde{\mathcal{E}}_{J_0,u,\epsilon}\rightarrow \widetilde{V}_u$, where $\widetilde{\mathcal{E}}_{J_0,u,\epsilon}$ has fiber $\Par_{uv}^{hol}\widetilde{E}_{J_0,u,\epsilon}$ at $v\in \widetilde{V}_u$, and $\varphi_{\ell,u}^{*}$ is defined in Definition \ref{Defi:5.21}.

  Let $\widetilde{h}\coloneqq \{\widetilde{h}_u\}$, and denote by
  \begin{align*}
    \widetilde{\bf ev}_k&\coloneqq \prod_{i=1}^k\ev_i:\widetilde{V}_u\rightarrow M^k,\\
    {\bf ev}_k&\coloneqq \prod_{i=1}^k\ev_i:V_w\rightarrow M^k
  \end{align*}
  for distinction. Note that the image $\widetilde{\bf ev}_k(\widetilde{h}^{-1}(0))$ coincides with ${\bf ev}_k(h^{-1}(0))$ as sets in $M^k$.
  
  For maps in $h^{-1}(0)\backslash {\bf ev}_k^{-1}(U_1)$, they are disjoint from $U_{w_i,T}$, $U_{Y,w_i,j}$, and $U_{Y,w_i,tan}$. Moreover, the images of their first $k$ marked points are at distance at least $\epsilon_1$ from $Y$. Therefore, we have
  \[ h_w^{-1}(0)\backslash {\bf ev}_k^{-1}(U_1) \subset V_{w,\epsilon_1}^{\bot},\]
  where $V_{w_i,\epsilon_1}^\bot$ is defined in Lemma \ref{lemma:5.25}. Furthermore,
  \[ \widetilde{h}_u^{-1}(0)\backslash \widetilde{\bf ev}_k^{-1}(U_1)=\varphi_\ell^{-1}(h_w^{-1}(0)\backslash {\bf ev}_k^{-1}(U_1)) \subset \varphi_\ell^{-1}(V_{w,\epsilon_1}^{\bot})=\widetilde{V}_{u,\epsilon_1}^{\bot}.\]
  This implies that, by Lemma \ref{lemma:5.25}, $\varphi_\ell$ is a submersion at maps in a neighborhood of $\widetilde{h}_u^{-1}(0)\backslash \widetilde{\bf ev}_k^{-1}(U_1)$. Thus $\widetilde{h}_u=\varphi_{\ell,u}^{*}(h_w)$ is transverse to $0$ in a neighborhood of $\widetilde{h}_u^{-1}(0)\backslash \widetilde{\bf ev}_k^{-1}(U_1)\subset \widetilde{V}_u$.

  Let $\mathfrak{C}$ be a compact subset of $\bigcup_u\widetilde{V}_u$ such that
  \[\widetilde{h}_u^{-1}(0)\backslash \widetilde{\bf ev}_k^{-1}(U)\subset\mathfrak{C}\cap \widetilde{V}_u\subset \widetilde{h}_u^{-1}(0)\backslash \widetilde{\bf ev}_k^{-1}(U_1)\]
  for every $u\in \overline{\mathcal{M}}_{0,k+\ell}(J_0,A,Y)$. By the arguments above, $\widetilde{h}_u$ is transverse to $0$ in a neighborhood of $\mathfrak{C}\cap \widetilde{V}_u$. Note also that maps in a sufficiently small neighborhood of $\mathfrak{C}\cap \widetilde{V}_u$ have only one node and are transverse to $Y$. Thus the map $\ev_{\mathbb{C},u}$ is also transverse to zero on $\mathfrak{C}\cap \widetilde{V}_u$.
  
  According to Lemma \ref{Lemma:5.4}, there is a multisection $\mwh\coloneqq \{\mwh_u\}$ such that
  \begin{enumerate}
    \item $\mwh_u=\widetilde{h}_u\times \ev_{\mathbb{C},u}$ in a neighborhood of $\mathfrak{C}\cap \widetilde{V}_u$ for each $u$.
    \item $\mwh_u$ is transverse to $0$.
    \item $\mwh_u$ is sufficiently close to $\widetilde{h}_u\times \ev_{\mathbb{C},u}$, such that
          \[\bigcup_{u\in \overline{\mathcal{M}}_{0,k+\ell}(J_0,A,Y)}\left(\widetilde{\bf ev}_k ((\widetilde{h}_u\times \ev_{\mathbb{C},u})^{-1}(0))\bigtriangleup \widetilde{\bf ev}_k (\mwh_u^{-1}(0))\right)\subset U,\]
          where $\bigtriangleup$ denotes the symmetric difference of sets.
  \end{enumerate}

  \textbf{Step 4. \textit{Prove that $\varphi_\ell$ is a $\ell!$ covering map on $\mathfrak{C}$.}} Take a triangulation
  \[\mwh^{-1}(0)=\bigcup_{\widetilde{\aleph}=1}^{\widetilde{\mathfrak{m}}}\widetilde{\sigma}_{\widetilde{\aleph}}(\Delta_d),\]
  where each $\widetilde{\sigma}_{\widetilde{\aleph}}(\Delta_d)$ is sufficiently small such that $\widetilde{\sigma}_{\widetilde{\aleph}}(\Delta_d)\subset \widetilde{V}_u$ for some $u$. Moreover, we require that if $\widetilde{\sigma}_{\widetilde{\aleph}}(\Delta_d)\cap \mathfrak{C}\neq\emptyset$ then
  \begin{equation}
    \label{eq:5.1.1}
    \widetilde{\sigma}_{\widetilde{\aleph}}(\Delta_d)\cap \widetilde{\bf ev}_k^{-1}(U_1)=\emptyset,
  \end{equation}
  where $U_1$ is defined in Step 2. For any $\widetilde{\sigma}_{\widetilde{\aleph}}(\Delta_d)$ does not intersect $\mathfrak{C}$ and $v\in\widetilde{\sigma}_{\widetilde{\aleph}}(\Delta_d)$, let $\hat{w}\coloneqq \varphi_\ell(v)$. Notice that since $\hat{w}$ is away from $U_{w,T}$, $U_{Y,w,j}$ and $U_{Y,w,tan}$, the underlying curve of $\hat{w}$ is simply $S^2$. Moreover $\hat{w}$ intersects $Y$ transversely, and the images of the $k$ marked points of $\hat{w}$ are away from $Y$. Therefore for any $v\in\varphi_\ell^{-1}(\hat{w})$, the map $v$ has only one node and intersects $Y$ transversely as well, the images of the marked points of $v$ are pairwise distinct. This implies that {\bf $\varphi_\ell^{-1}(\hat{w})$ contains $\ell!$ maps} due to the permutation of the last $\ell$ marked points. We also take a triangulation
  \[h^{-1}(0)=\bigcup_{\aleph=1}^{\mathfrak{m}}\sigma_\aleph(\Delta_d),\]
  each $\sigma_\aleph(\Delta_d)$ is sufficiently small such that, if $\varphi_\ell^{-1}(\sigma_\aleph(\Delta_d))\cap \mathfrak{C}\neq\emptyset$, then $\varphi_\ell^{-1}(\sigma_\aleph(\Delta_d))$ is contained in the union of some $\widetilde{\sigma}_{\widetilde{\aleph}}(\Delta_d)$ satisfying $\widetilde{\sigma}_{\widetilde{\aleph}}(\Delta_d)\cap \mathfrak{C}\neq \emptyset$. Therefore, for each $\hat{w}$ in such a $\sigma_\aleph(\Delta_d)$, the set $\varphi_\ell^{-1}(\hat{w})$ contains $\ell!$ distinct maps.

  Now take any $\widetilde{\sigma}_{\widetilde{\aleph}}(\Delta_d)$ with $\widetilde{\sigma}_{\widetilde{\aleph}}(\Delta_d)\cap\mathfrak{C}\neq\emptyset$, and any $\sigma_\aleph(\Delta_d)$ satisfying
  \[\sigma_\aleph(\Delta_d)\cap \mathfrak{C}\neq\emptyset,\quad\sigma_\aleph(\Delta_d)\cap\varphi_\ell(\widetilde{\sigma}_{\widetilde{\aleph}}(\Delta_d))\neq\emptyset.\]
  Let $v\in \varphi_\ell^{-1}(\sigma_\aleph(\Delta_d))\cap\widetilde{\sigma}_{\widetilde{\aleph}}(\Delta_d)$, then $\varphi_\ell$ is a diffeomorphism from a neighborhood of $v\in \widetilde{\sigma}_{\widetilde{\aleph}}(\Delta_d)$ to a neighborhood of $\hat{w}=\varphi_\ell(v)\in \sigma_\aleph(\Delta_d)$. According to (\ref{eq:4.1.1}), there are the following isomorphisms
  \begin{equation}
    \label{eq:5.3.3}
    \begin{aligned}
      &\det T\widetilde{V}_u\otimes \det \widetilde{\mathcal{E}}_{J_0,u,\epsilon}^{*}\cong \det T\widetilde{\sigma}_{\widetilde{\aleph}}(\Delta_d)\\
      &\det TV_w\otimes\det \mathcal{E}_{J_0,w,\epsilon}^{*}\cong \det T\sigma_\aleph(\Delta_d).
    \end{aligned}
  \end{equation}
  Since $\varphi_\ell$ is a diffeomorphism in a neighborhood of $v$, the map $\varphi_{\ell*}:T\widetilde{\sigma}_{\widetilde{\aleph}}(\Delta_d)\rightarrow T\sigma_\aleph(\Delta_d)$ is an isomorphism. Together with (\ref{eq:5.3.3}), we have the following isomorphism on some neighborhoods $O_v$ and $O_{\hat{w}}$ of $v$ and $\hat{w}$ respectively,
  \begin{equation}
    \label{eq:5.4.4}
    (\det T\widetilde{V}_u\otimes \det \widetilde{\mathcal{E}}_{J_0,u,\epsilon}^{*})|_{O_v}\cong (\det TV_w\otimes\det \mathcal{E}_{J_0,w,\epsilon}^{*})|_{O_{\hat{w}}}.
  \end{equation}
  
  As discussed on \cite[Prop.~16.5]{FK}, there exist canonical orientations for the Kuranishi structures on $\overline{\mathcal{M}}_{0,k}(J_0,A)$ and $\overline{\mathcal{M}}_{0,k+\ell}(J_0,A,Y)$. Let us examine equation (\ref{eq:5.4.4}) fiberwise. At $v$ we have
  \begin{equation}
    \label{eq:5.3.4}
    \begin{aligned}
    &\det T_v\widetilde{V}_u\otimes \det (\Par^{hol}_{uv}\widetilde{E}_{J_0,u,\epsilon}\times\mathbb{C}^\ell)^{*} \\
    \cong&\det T_v\widetilde{V}_{u,map}\otimes\det T_v\widetilde{V}_{u,resolve}\otimes\det T_v\widetilde{V}_{u,deform} \\
    &\otimes\det (\Par^{hol}_{uv}\widetilde{E}_{J_0,u,\epsilon})^{*}\otimes\det(\mathbb{C}^\ell)^{*}\\
    \cong&\det T_v\widetilde{V}_{u,map}\otimes\det (\Par^{hol}_{uv}\widetilde{E}_{J_0,u,\epsilon})^{*} \\
    &\otimes\det T_v\widetilde{V}_{u,resolve}\otimes\det T_v\widetilde{V}_{u,deform}\otimes\det(\mathbb{C}^\ell)^{*}.
    \end{aligned}
  \end{equation}
  The term $\det T_v\widetilde{V}_{u,map}\otimes\det (\Par^{hol}_{uv}\widetilde{E}_{J_0,u,\epsilon}))^{*}$ corresponds to the determinant of $D_u\bar{\partial}_{J_0}$. Similarly,
  \begin{equation}
    \label{eq:5.3.5}
    \begin{aligned}
      &\det T_{\hat{w}}V_{w}\otimes\det (\Par^{hol}_{w\hat{w}}E_{J_0,w,\epsilon}))^{*} \\
      \cong&\det T_{\hat{w}}V_{w,map}\otimes\det (\Par^{hol}_{w\hat{w}}E_{J_0,w,\epsilon}))^{*} \\
      &\otimes\det T_{\hat{w}}V_{w,resolve}\otimes\det T_{\hat{w}}V_{w,deform},
    \end{aligned}
  \end{equation}
  where $\det T_{\hat{w}}V_{w,map}\otimes\det (\Par^{hol}_{w\hat{w}}E_{J_0,w,\epsilon}))^{*}$ corresponds to the determinant of $D_w\bar{\partial}_{J_0}$. Since on the obstruction bundles $\widetilde{\mathcal{E}}_{J_0,u,\epsilon}$ and $\mathcal{E}_{J_0,w,\epsilon}$ we have
  \[\bar{\partial}_{J_0}\circ \varphi_\ell=\varphi_\ell\circ \bar{\partial}_{J_0},\]
  it follows that $\varphi_\ell$ preserves orientations on $\det T_v\widetilde{V}_{u,map}\otimes\det (\Par^{hol}_{uv}\widetilde{E}_{J_0,u,\epsilon}))^{*}$ and $\det T_{\hat{w}}V_{w,map}\otimes\det (\Par^{hol}_{w\hat{w}}E_{J_0,w,\epsilon}))^{*}$. 

  Furthermore, the map $v\in \widetilde{\sigma}_{\widetilde{\aleph}}(\Delta_d)$ has exactly one node and intersects $Y$ transversely. The map $\hat{w}\coloneqq \varphi_\ell(v)$ also has exactly one node and intersects $Y$ transversely. Let $S^2_v$ denote the underlying sphere of $v$, and $S^2_{\hat{w}}$ the underlying sphere of $\hat{w}$. A small neighborhood $O_{S^2_v}\subset \widetilde{V}_{u,resolve}\times \widetilde{V}_{u,deform}$ of $S^2_v$ is biholomorphic to $\mathbb{C}^{k+\ell}$, representing the deformation of all $k+\ell$ marked points on $S^2_v$. Similarly, a small neighborhood $O_{S^2_{\hat{w}}}\subset V_{w,resolve}\times V_{w,deform}$ of $S^2_{\hat{w}}$ is biholomorphic to $\mathbb{C}^k$, representing the deformation of all $k$ marked points on $S^2_{\hat{w}}$. 
  
  Since $v$ is transverse to $Y$, and moreover according to the construction of the obstruction data $\widetilde{E}_{J_0,u,\epsilon}$, the map $v$ is $J_0$-holomorphic near its intersections with $Y$. Therefore $\prod_{i=1}^\ell\ev_{k+i}$ maps $\mathbb{C}^\ell\subset \mathbb{C}^{k+\ell}\cong O_{S^2_v}$ $J_0$-holomorphically to $M^k$. Notably, the map $\ev_{\mathbb{C},u}$ is $\prod_{i=1}^\ell\ev_{k+i}$ followed by a projection to the normal direction of $Y$ in $M$. Since $Y$ is a $J_0$-holomorphic hypersurface, the projection is $J_0$-holomorphic as well. Thus $\ev_{\mathbb{C},u}$ maps the last $\ell$ factors $\mathbb{C}^\ell\subset O_{S^2_v}$ $J_0$-holomorphically to the $\mathbb{C}^\ell$ component in the obstruction data $\widetilde{E}_{J_0,u,\epsilon}\times \mathbb{C}^\ell$. Therefore, for such a $v\in \widetilde{\sigma}_{\widetilde{\aleph}}(\Delta_d)$,
  \begin{equation}
    \label{eq:5.3.6}
    \begin{aligned}
      &\det T_v\widetilde{V}_{u,resolve}\otimes\det T_v\widetilde{V}_{u,deform}\otimes\det(\mathbb{C}^\ell)^{*}\\
      \cong& \det(T_{S^2_v}O_{S^2_v})\otimes\det(\mathbb{C}^\ell)^{*}\\
      \cong& \det(\mathbb{C}^k) \cong \det(T_{S^2_{\hat{w}}}O_{S^2_{\hat{w}}})\\
      \cong& \det T_{\hat{w}}V_{w,resolve}\otimes\det T_{\hat{w}}V_{w,deform},
    \end{aligned}
  \end{equation}
  where $T_{S^2_v}O_{S^2_v}$ and $T_{S^2_{\hat{w}}}O_{S^2_{\hat{w}}}$ are the tangent spaces at $S^2_v$ and $S^2_{\hat{w}}$ respectively. The third line of (\ref{eq:5.3.6}) follows from the fact that the first $k$ factors of $\mathbb{C}^{k+\ell}\cong O_{S^2_v}$ represent precisely the deformation of the first $k$ marked points on $S^2_v$, which is identical (via the map $\varphi_\ell$) to the deformation of the $k$ marked points on $S^2_{\hat{w}}$. All isomorphisms in (\ref{eq:5.3.6}) are orientation-preserving.

  Now combine (\ref{eq:5.3.4}), (\ref{eq:5.3.5}), and the following two facts:
  \begin{itemize} 
    \item $\varphi_\ell$ preserves orientations on $\det T_v\widetilde{V}_{u,map}\otimes\det (\Par^{hol}_{uv}\widetilde{E}_{J_0,u,\epsilon})^{*}$ \\
     and $\det T_{\hat{w}}V_{w,map}\otimes\det (\Par^{hol}_{w\hat{w}}E_{J_0,w,\epsilon})^{*}$,
    \item all isomorphisms in (\ref{eq:5.3.6}) are orientation preserving,
  \end{itemize}
  we conclude that the isomorphism (\ref{eq:5.4.4}) is orientation preserving. Combining everything in this step, we prove that {\bf $\varphi_\ell$ is a $\ell!$ covering map on $\mathfrak{C}$}. 

  \textbf{Step 5. \textit{Complete the proof.}} Now we apply A.Zinger's results in \cite{Zinger} here. The open neighborhood $U$ satisfies that
  \[0\to H_d(M^k;\mathbb{Q})\to H_d(M^k,U;\mathbb{Q})\to 0\]
  as in (\ref{eq:3.3.2}), and on the compact subset $\mathfrak{C}$, the forgetful map $\varphi_\ell$ is a $\ell!$ covering map. Then we have the following calculations:
  \begin{align*}
    &{\bf ev}_{k*}([\widetilde{\mathfrak{h}}^{-1}(0)])-[{\bf ev}_k:\mathfrak{C}\to M^k]=0\in H_d(M^k,U;\mathbb{Q})\\
    &\ell!{\bf ev}_{k*}([h^{-1}(0)])-[{\bf ev}_k:\mathfrak{C}\to M^k]=0\in H_d(M^k,U;\mathbb{Q})
  \end{align*}
  
  Thus the two virtual fundamental classes we are\label{Fix:M:6:3} trying to compare differ by the factor $\ell!$ in $H_d(M^k,U;\mathbb{Q})\cong H_d(M^k;\mathbb{Q})$ and this completes the proof.
\end{proof}

\subsection{Comparison between invariants}\label{Sec:5.5}
In this section, we piece together the results from Section \ref{Sec:5.1} to Section \ref{Sec:PullBack}. Recall the spaces $B^{*}$, $\mathcal{J}_{\ell+1}^{*}(M,Y;J,\theta_1)$ defined in Definition \ref{Def:3.20}, and the space $\mathcal{J}_{\ell+1}^{reg*}\subset \mathcal{J}_{\ell+1}^{*}(M,Y;J,\theta_1)$ defined in Theorem \ref{Thm:3.25}. Take a $K\in \mathcal{J}_{\ell+1}^{reg*}$ and a domain independent $J_0\in B^{*}\cap \mathcal{J}_{\ell+1}^{*}(M,Y;J,\theta_1)$, such that $K, J_0$ are connected by a path $\{K_t\}_{t\in[0,1]}\subset \mathcal{J}_{\ell+1}^{*}(M,Y;J,\theta_1)$ with $K_0=J_0$, $K_1=K$. We shall prove the final result:

\begin{theorem}[Theorem \ref{introcorollayA}]\label{Thm:5.17}
  For $K\in \mathcal{J}_{\ell+1}^{reg*}$ and $J_0\in B^{*}\cap \mathcal{J}_{\ell+1}^{*}(M,Y;J,\theta_1)$ connected by a path $\{K_t\}_{t\in[0,1]}\subset \mathcal{J}_{\ell+1}^{*}(M,Y;J,\theta_1)$ with $K_0=J_0$, $K_1=K$, take a generic multisection $s^{'}=\{s^{'}_w\}$ on the Kuranishi structure
  \[\{(V_w,E_{J_0,w},\Gamma_w,\psi_w,s_w)\}_{w\in \overline{\mathcal{M}}_{0,k}(J_0,A)}.\]
  Then the homology class given by the pseudocycle
  \[{\bf ev}_k\coloneqq \prod_{i=1}^k\ev_i:\mathcal{M}_{0,k+\ell}(K,A,Y)\rightarrow M^k\]
  is equal to $\ell !$ times the homology class ${\bf ev}_{k*}([s^{'-1}(0)])$.
\end{theorem}
\begin{proof}
  By theorem \ref{Thm:5.12}, there exists a Kuranishi structure
  \[\{(V_u,E_{K,u}\times\mathbb{C}^\ell,\Gamma_u,\psi_u,s_u\times \ev_{\mathbb{C},u})\}_{u\in\overline{\mathcal{M}}_{0,k+\ell}(K,A,Y)}\]
  over $\overline{\mathcal{M}}_{0,k+\ell}(K,A,Y)$. Take a generic multisection $h_K=\{h_{u;K}\}$ of the obstruction bundle, according to Theorem \ref{Thm:5.14}\label{Fix:C:3:8} (also notice that the virtual fundamental class does not depend on the choice of a generic multisection),
  \[[{\bf ev}_k:\mathcal{M}_{0,k+\ell}(K,A,Y)\to M^k]={\bf ev}_{k*}([h_K^{-1}(0)]),\]
  $[{\bf ev}_k:\mathcal{M}_{0,k+\ell}(K,A,Y)\to M^k]$ denotes the homology class induced by the pseudocycle ${\bf ev}_k$. According to Theorem \ref{Thm:5.15}, there is a cobordism between the Kuranishi structures over $\overline{\mathcal{M}}_{0,k+\ell}(K,A,Y)$ and $\overline{\mathcal{M}}_{0,k+\ell}(J_0,A,Y)$. Therefore
  \[{\bf ev}_{k*}([h_K^{-1}(0)])={\bf ev}_{k*}([h_{J_0}^{-1}(0)]),\]
  where $h_{J_0}$ is a generic multisection on
  \[\{(V_u,E_{J_0,u}\times\mathbb{C}^\ell,\Gamma_u,\psi_u,s_u\times \ev_{\mathbb{C},u})\}_{u\in\overline{\mathcal{M}}_{0,k+\ell}(J_0,A,Y)}.\]

  Now we shift to the Kuranishi structure
  \[\{(\widetilde{V}_u,\widetilde{E}_{J_0,u,\epsilon}\times \mathbb{C}^\ell,\Gamma_u,\widetilde{\psi}_{u},\widetilde{s}_u\times ev_{\mathbb{C},u})\}_{u\in \overline{\mathcal{M}}_{0,k+\ell}(J_0,A,Y)},\]
  and take a generic multisection $\widetilde{h}_\epsilon$ on it. We also take a generic multisection $h_\epsilon$ on
  \[\{(V_w,E_{J,w,\epsilon},\Gamma_w,\psi_w,s_w)\}_{w\in\overline{\mathcal{M}}_{0,k}(J_0,A)}.\]
  According to Theorem \ref{Thm:5.22},
  \[\ell !{\bf ev}_{k*}([h_\epsilon^{-1}(0)])={\bf ev}_{k*}([\widetilde{h}_\epsilon^{-1}(0)]).\]

  Finally, take a generic multisection $s^{'}$ on
  \[\{(V_w,E_{J_0,w},\Gamma_w,\psi_w,s_w)\}_{w\in \overline{\mathcal{M}}_{0,k}(J_0,A)}.\]
  Since the virtual fundamental class does not depend on the choice of Kuranishi structures, we have
  \begin{align*}
    [{\bf ev}_k:\mathcal{M}_{0,k+\ell}(K,A,Y)\to M^k]=&{\bf ev}_{k*}([h_K^{-1}(0)])={\bf ev}_{k*}([h_{J_0}^{-1}(0)])\\
    =&{\bf ev}_{k*}([\widetilde{h}_\epsilon^{-1}(0)])=\ell !{\bf ev}_{k*}([h_\epsilon^{-1}(0)])\\
    =&\ell !{\bf ev}_{k*}([s^{'-1}(0)]).
  \end{align*}
  This completes the proof.
\end{proof}

\bibliographystyle{plain}

\end{document}